\documentclass[12pt]{amsart}
\usepackage{amssymb}
\input amssym.def
\usepackage{amsmath, amsfonts,hyperref}
\usepackage{amscd}
\usepackage[mathscr]{eucal}

\usepackage{xcolor}
\hypersetup{
    colorlinks,
    linkcolor={red!50!black},
    citecolor={blue!50!black},
    urlcolor={blue!80!black}
}

\newcommand{\N}{{\mathbb N}}
\newcommand{\Z}{{\mathbb Z}}
\newcommand{\Q}{{\mathbb Q}}
\newcommand{\R}{{\mathbb R}}
\newcommand{\C}{{\mathbb C}}

\newcommand{\cE}{{\mathcal E}}
\newcommand{\cF}{{\mathcal F}}
\newcommand{\cH}{{\mathcal H}}
\newcommand{\cM}{{\mathcal M}}
\newcommand{\cO}{{\mathcal O}}
\newcommand{\cS}{{\mathcal S}}
\newcommand{\z}{\zeta}
\newcommand{\reg}{{\rm Reg}}
\newcommand{\ord}{{\rm ord}}

\newtheorem{thm}{Theorem}
\newtheorem{lem}{Lemma}
\newtheorem{cor}{Corollary}
\newtheorem{prop}{Proposition}
\newtheorem{rmk}{Remark}
\newtheorem{exm}{Example}
\newtheorem{defn}{Definition}
\newcommand{\thmref}[1]{Theorem~\ref{#1}}
\newcommand{\propref}[1]{Proposition~\ref{#1}}
\newcommand{\lemref}[1]{Lemma~\ref{#1}}
\newcommand{\corref}[1]{Corollary~\ref{#1}}
\newcommand{\rmkref}[1]{Remark~\ref{#1}}
\newcommand{\exmref}[1]{Example~\ref{#1}}

\DeclareFontFamily{U}{wncy}{}
\DeclareFontShape{U}{wncy}{m}{n}{<->wncyr10}{}
\DeclareSymbolFont{mcy}{U}{wncy}{m}{n}
\DeclareMathSymbol{\Sh}{\mathord}{mcy}{"58}

\begin{document}

\title[Laurent type expansion of multiple zeta functions]
{Multiple Stieltjes constants and Laurent type expansion of the multiple
zeta functions at integer points}

\author{Biswajyoti Saha}

\address{Biswajyoti Saha\\ \newline
School of Mathematics and Statistics, University of Hyderabad,
Prof. C.R. Rao Road, Gachibowli, Hyderabad - 500046, India}
\email{biswa.imsc@gmail.com, biswa@uohyd.ac.in}



\date{\today}

\begin{abstract}
In this article, we study the local behaviour of the multiple zeta functions at integer points
and write down a Laurent type expansion of the multiple zeta functions around these points.
Such an expansion involves a convergent power series whose coefficients are obtained by
a regularisation process, similar to the one used in defining the classical Stieltjes constants for
the Riemann zeta function. We therefore call these coefficients {\it multiple Stieltjes constants}.
The remaining part of the above mentioned Laurent type expansion is then expressed in terms
of the multiple Stieltjes constants arising in smaller depths.
\end{abstract}

\maketitle

\section{Introduction}\label{intro}

Throughout the paper, a natural number will mean a non-negative integer
and their set will be denoted by $\N$.
Let $r$ be a natural number. The multiple zeta function of depth $r$
is the holomorphic function defined in the open set
$$
U_r :=\{(s_1,\ldots,s_r) \in \C^r  : \Re(s_1 + \cdots +s_i)>i \ \text{for} \ 1 \le i \le r \},
$$
by the series expression :
\begin{equation}\label{zeta}
\z(s_1,\ldots,s_r) := \sum_{n_1 > \cdots >n_r>0} n_1^{-s_1} \cdots n_r^{-s_r},
\end{equation}
which converges normally on any compact subset of $U_r$. In particular,
the multiple zeta function of depth $0$ is defined by $\zeta(\varnothing) :=1$.
The meromorphic continuation of the multiple zeta functions is now well known.
This was first established by Zhao \cite{JZ}. The exact set of singularities
was identified by Akiyama, Egami and Tanigawa \cite{AET};
the polar hyperplanes are simple and given by the following equations :
\begin{equation*}
\begin{split}
& s_1=1, \ \text{if} \ r \ge 1;\\
& s_1+s_2=2,1,0,-2,-4,-6, \ldots,\ \text{if} \ r \ge 2;\\
& s_1+\cdots+s_i = i-n, \ \text{ for all integers} \
n,i \ \text{such that} \ n \ge 0 \ \text{and} \ 3 \le i \le r.
\end{split}
\end{equation*}

In this paper, we define for each $(a_1,\ldots,a_r) \in \Z^r$,
and $(k_1,\ldots,k_r) \in \N^r$, a regularised value
$\gamma_{k_1, \ldots, k_r}^{(a_1,\ldots,a_r)}$ for the (not
necessarily convergent) series
\begin{equation}\label{zeta-der-series}
\sum_{n_1 > \cdots >n_r>0}
\frac{\log^{k_1}n_1 \cdots \log^{k_r} n_r}{n_1^{a_1} \cdots n_r^{a_r}}.
\end{equation}
When $(a_1,\ldots,a_r) \in U_r$, the above series converges absolutely and its sum is
$$
(-1)^{k_1+ \cdots + k_r}D^{(k_1,\ldots,k_r)}\z(a_1,\ldots,a_r).
$$
In this case $\gamma_{k_1, \ldots, k_r}^{(a_1,\ldots,a_r)}$
is defined to be this sum. For the general case, we consider the 
truncated finite series
\begin{equation}\label{zeta-der-trunc-series}
\sum_{N > n_1 > \cdots >n_r>0} \frac{\log^{k_1}n_1 \cdots \log^{k_r} n_r}
{n_1^{a_1} \cdots n_r^{a_r}},
\end{equation}
and we show that, as a function of the integer $N$, it has the form
$P(\log N,N) + o(1)$ when $N \to \infty$, where $P$ is
a polynomial in two indeterminates with coefficients in $\Q$.
This polynomial is uniquely determined by $(a_1,\ldots,a_r)$
and $(k_1,\ldots,k_r)$, and
$\gamma_{k_1, \ldots, k_r}^{(a_1,\ldots,a_r)}$ is defined to be
its constant term. In the special case $(a_1,\ldots,a_r)=(1,\ldots,1)$,
we simply denote this number by $\gamma_{k_1, \ldots, k_r}$.

\begin{rmk}\label{rmk-boundary-poly}\rm
If $(a_1,\ldots,a_r) \in \partial U_r$, then the above polynomial
is in fact a polynomial in $\log N$ only (see \rmkref{rmk-order-3} below). 
\end{rmk}

\begin{rmk}\label{rmk-msc}  \rm
In fact, we shall prove that there exists a Laurent series
$F=\sum_{n} F_n(L) X^n \in \Q[L]((X))$ (where $F_n=0$ for
sufficiently small $n$), such that \eqref{zeta-der-trunc-series}
has an asymptotic expansion
$$
\sum_{n \le A} F_n(\log N) N^{-n} + o(N^{-A}),
$$
as $N \to \infty$, for any $A \in \N$.
\end{rmk}

The numbers $\gamma_{k_1, \ldots, k_r}^{(a_1,\ldots,a_r)}$ are called
the {\it multiple Stieltjes constants}
(of order $(k_1, \ldots, k_r)$ at the point $(a_1,\ldots,a_r)$,
when this needs to be specified), as they are nothing but the
classical Stieltjes constants
$$
\gamma_k :=\lim_{N \to \infty} \left(\sum_{1 \le n < N} \frac{\log^k n}{n} -
\frac{\log^{k+1} N}{k+1} \right),
$$
in the particular case when $r=1,a_1=1,k_1=k$. More detailed
discussion about these constants is given in Section \ref{msc}.

It is a classical result (due to Stieltjes (1885), see \cite[Letter 75]{BB})
that the Riemann zeta function has the following Laurent series expansion around $1$ :
\begin{equation}\label{zeta-exp}
\z(s)= \frac{1}{s-1} + \sum_{k\ge 0} \frac{(-1)^k \gamma_k}{k!} (s-1)^k,
\end{equation}
and moreover the series on the right hand side converges on the
whole of $\C$.

Our goal in this paper is to find a similar Laurent type expansion
for $\z(s_1,\ldots,s_r)$ around any integer point $(a_1,\ldots,a_r) \in \Z^r$,
and to explicitly express its coefficients in terms of
the multiple Stieltjes constants. For this purpose we
consider the following formal power series
\begin{equation}\label{reg-multi-zeta}
\sum_{k_1,\ldots,k_r \ge 0}
\frac{(-1)^{k_1+\cdots+k_r}}{k_1! \cdots k_r!}
\gamma_{k_1, \ldots, k_r}^{(a_1,\ldots,a_r)}
(s_1-a_1)^{k_1} \cdots (s_r-a_r)^{k_r}.
\end{equation}
We prove that it converges in a neighbourhood of $(a_1, \ldots, a_r)$
and extends to a meromorphic function in the whole of $\C^r$. We call
this meromorphic function {\it the regularised multiple zeta function around
$(a_1, \ldots, a_r)$} and denote it by
$\z^\reg_{(a_1, \ldots, a_r)} (s_1,\ldots,s_r)$.

Of course, when $(a_1,\ldots,a_r) \in U_r$, \eqref{reg-multi-zeta}
is the Taylor expansion of $\z(s_1,\ldots,s_r)$
at the point $(a_1,\ldots,a_r)$ and hence in this case
\begin{equation}\label{multi-zeta-conv}
\z(s_1,\ldots,s_r)=\z^\reg_{(a_1, \ldots, a_r)} (s_1,\ldots,s_r),
\end{equation}
as meromorphic functions on $\C^r$. But this is no more true in general.
As an example, formula \eqref{zeta-exp} can be restated as
$$
\z(s)= \frac{1}{s-1} + \z^\reg_{(1)}(s). 
$$
In this paper, we extend to all multiple zeta functions and all 
integer points this type of relation between the multiple zeta
functions and their regularised counterparts.

\begin{rmk}\label{rmk-mzf-non-conv}   \rm
There have been numerous research on assigning suitable
values to multiple zeta functions at integer points outside
the domain of convergence, for example see \cite{AET,AT,TO}.
Our approach allows us to understand completely the local
behaviour of these functions near those points, and to
recover such results.
\end{rmk}

This paper is organised according to the increasing level of difficulty :
we treat the case when $(a_1,\ldots,a_r)= (1,\ldots,1)$ in Section \ref{boundary-1}
(see \thmref{thm-multi-zeta-exp}). In Section \ref{pos-boundary}, we consider
a more general case of boundary points of $U_r$ of a particular form (see
\thmref{thm-gen-multi-zeta-exp}). For instance, \thmref{thm-gen-multi-zeta-exp} is
applicable for boundary points of $U_r$ with positive integral coordinates.
In both Theorems \ref{thm-multi-zeta-exp} and \ref{thm-gen-multi-zeta-exp},
we give explicit expressions of the multiple zeta function
$\z(s_1,\ldots,s_r)$ in terms of the regularised multiple zeta functions
$\z^\reg_{(a_i, \ldots, a_r)} (s_i,\ldots,s_r)$ for $1 \le i \le r$.

In Section \ref{gen-boundary}, we consider the case of general
boundary points of $U_r$ with integral coordinates. In this case,
it appears to be more natural and convenient to express the regularised
multiple zeta function $\z^\reg_{(a_1, \ldots, a_r)} (s_1,\ldots,s_r)$
in terms of $\z(s_i,\ldots,s_r)$ for $1 \le i \le r$
(see \thmref{thm-reg-exp}). We then need an inversion process,
explained in Section \ref{sec-inv}, to get a Laurent type expansion of $\z(s_1,\ldots,s_r)$
around $(a_1,\ldots,a_r)$ and thereby we also recover the previous results
(see \thmref{thm-reg-exp-inverse}).

Finally, in Section \ref{gen-point}, we extend these results to all integer points
$(a_1,\ldots,a_r) \in \Z^r$ (see \thmref{thm-gen-reg-exp}).

Each of the Theorems \ref{thm-multi-zeta-exp}, \ref{thm-gen-multi-zeta-exp},
\ref{thm-reg-exp} and \ref{thm-gen-reg-exp} of course implies the 
preceding ones. But since their formulation varies and also the proofs get
more involved and require more machinery as we go along, we have included
independent proofs to keep our exposition reader friendly.

\section{Multiple Stieltjes constants}\label{msc}

In this section, we prove the existence of the asymptotic expansions of the 
type described in \rmkref{rmk-msc} and from this we deduce the definition
of the multiple Stieltjes constants.
To do this we use the language of asymptotic expansions of sequences of
complex numbers relative to a comparison scale, in the sense of
Bourbaki \cite[Chap V, \S2]{NB}.

The set $\cE$ of sequences 
$$
\left((\log n)^l n^{-m} \right)_{n\ge 1}
$$
where $l \in \N$ and $m \in \Z$, is a comparison scale on the set of
natural numbers $\N$, filtered by the Frechet filter (see \cite[Chap V, \S2, Def. 1]{NB}).
We say that a sequence of complex numbers $(u_n)_{n \in \N}$ has
an {\it asymptotic expansion to arbitrary precision relative to $\cE$}
if it has an asymptotic expansion to precision $n^{-A}$ for any integer $A$
(see \cite[Chap V, \S2, Def. 2]{NB}). This means that
there exists a formal Laurent series
$F=\sum_{l \in \N, m \in \Z} \lambda_{(l,m)} L^l X^m \in \C[L]((X))$
in the indeterminate $X$, with coefficients in the polynomial ring $\C[L]$,
such that for any integer $A$, we have
$$
u_n - \sum_{l \ge 0, m \le A}
\lambda_{(l,m)} (\log n)^l n^{-m} = o(n^{-A}),
$$
as $n \to \infty$.

\begin{defn}
When the above condition is satisfied, the Laurent series $F$ is unique
and we call it {\rm the formal asymptotic expansion of the sequence
$(u_n)_{n \in \N}$ (relative to $\cE$)}.
The constant term $\lambda_{(0,0)}$ is then called {\rm the regularised
value of the sequence $(u_n)_{n \in \N}$ (relative to $\cE$)}.
\end{defn}

\begin{rmk}\label{rmk-order-1} \rm
Note that, by definition of the ring $\C[L]((X))$, the coefficients of
$F$ have the following properties : there exists $m_0 \in \Z$ such that
$\lambda_{(l,m)}=0$ if $m < m_0$, and for any $m \in \Z$,
the set of integers $l \in \N$ such that $\lambda_{(l,m)} \neq 0$, is finite.
When $F=\sum_{m \in \Z}F_m(L)X^m \neq 0$, the smallest $m$ for which $F_m \neq 0$
is denoted by $\ord_X(F)$ and called {\it the order of $F$}. We then have
$u_n=O((\log n)^l n^{-m})$ where $m=\ord_X(F)$ and $l=\deg(F_m)$.
For $F=0$, we define $\ord_X(F)$ to be $\infty$.
\end{rmk}

If two sequences differ by only finitely many terms and one of them has
an asymptotic expansion to arbitrary precision relative to $\cE$,
then the other one also has such an expansion and their formal asymptotic
expansions are the same. This observation allows us to extend Definition 1 to
sequences $(u_n)$ which are only defined for $n$ large enough.

The set $\cS$ of sequences of complex numbers which have an asymptotic expansion
to arbitrary precision relative to $\cE$ is a unitary subalgebra of $\C^\N$
and the map which associates to such a sequence its formal
asymptotic expansion, is a $\C$-algebra homomorphism from $\cS$ to $\C[L]((X))$.

\begin{prop}\label{prop-partial-sum-asymp}
Let $(v_n)_{n \in \N}$ be a sequence of complex numbers which has an
asymptotic expansion to arbitrary precision relative to $\cE$. Then the sequence 
$(u_n)_{n \in \N}$ defined by $u_n  := \sum_{m=0}^{n-1} v_m$ also has
such an expansion.
\end{prop}

\begin{proof}
It is enough to show that the sequence $(u_n)_{n \in \N}$ has an
asymptotic expansion to precision $n^{-A}$ relative to $\cE$,
for any integer $A \ge 1$. By the hypothesis, the sequence $(v_n)_{n \in \N}$
has an asymptotic expansion
$$
v_n = \sum_{l \ge 0, m \le A+1}
\lambda_{(l,m)} (\log n)^l n^{-m} + o(n^{-A-1})
$$
to precision $n^{-A-1}$ relative to $\cE$, as $n \to \infty$.
Hence it is enough to prove \propref{prop-partial-sum-asymp}
in the following two cases :\\
a) when $v_n=(\log n)^l n^{-m}$ for $n \ge 1$, with $l \in \N,m \in \Z$,\\
b) when $v_n=o(n^{-A-1})$ as $n \to \infty$.

Note that derivatives and primitives of the functions on $(1,\infty)$
of the form $f_{(l,m)}(t)=(\log t)^l t^{-m}$, for $l \in \N$ and $m \in \Z$,
are finite $\Q$-linear combinations of functions of the same form.
Hence, Euler-Maclaurin summation formula yields asymptotic expansions
of $(u_n)_{n \in \N}$ to arbitrary precision in case a).

Next note that if $v_n=o(n^{-A-1})$ with $A \ge 1$,
the series $\sum_{m=0}^{\infty} v_m$ is absolutely convergent, and
if $s$ denotes its sum, then $u_n = s + o(n^{-A})$.
This completes the proof of \propref{prop-partial-sum-asymp}.
\end{proof}

\begin{rmk}\label{rmk-order-2} \rm
The proof also yields the following result : if $a$ is the order of
the formal asymptotic expansion of the sequence $(v_n)$, then the order
of the formal asymptotic expansion of the sequence $(u_n)$ is at least
$\min(0,a-1)$.
\end{rmk}

\begin{thm}\label{thm-msc}
For any $(a_1,\ldots,a_r) \in \Z^r$ and any $(k_1,\ldots,k_r) \in \N^r$,
the sequence $(u_N)_{N \ge 1}$ defined by
$$
u_N  :=\sum_{N > n_1 > \cdots >n_r>0} \frac{\log^{k_1}n_1 \cdots \log^{k_r} n_r}
{n_1^{a_1} \cdots n_r^{a_r}}
$$
has an asymptotic expansion to arbitrary precision relative to $\cE$.
\end{thm}

\begin{proof}
We prove this theorem by induction on $r$. It is clear for $r=0$.
Now we assume $r \ge 1$. Let $(v_n)_{n \ge 1}$ denote the sequence
defined by
$$
v_n=\sum_{n > n_2 > \cdots >n_r>0} \frac{\log^{k_2}n_2 \cdots \log^{k_r} n_r}
{n_2^{a_2} \cdots n_r^{a_r}}
$$
and $(w_n)_{n \ge 1}$ denote the sequence defined by
$w_n = \frac{\log^{k_1}n}{n^{a_1}}$. By definition of $\cE$,
$(w_n)_{n \ge 1}$ has an asymptotic expansion to arbitrary precision
relative to $\cE$. The sequence $(v_n)_{n \ge 1}$ also has such
an expansion by the induction hypothesis. Since $u_N=\sum_{n <N} v_n w_n$,
we get that $(u_N)_{N \ge 1}$ has such an expansion by \propref{prop-partial-sum-asymp}.
\end{proof}

\begin{rmk}\label{rmk-order-3} \rm
Using \rmkref{rmk-order-2}, we get that the order of the formal asymptotic
expansion relative to $\cE$ of the sequence $(u_N)$ considered in \thmref{thm-msc}
is at least $\min(0,a_1-1,\ldots,a_1+\cdots+a_r-r)$. In particular, when
$(a_1,\ldots,a_r)$ belongs to the closure $\overline{U_r}$ of $U_r$, this order
is non-negative, and therefore there exists a polynomial $P \in \C[L]$ such that
$u_N=P(\log N)+o(1)$ as $n$ tends to $\infty$.

\end{rmk}

\begin{defn}\label{defn-msc}
For any $(a_1,\ldots,a_r) \in \Z^r$ and any $(k_1,\ldots,k_r) \in \N^r$,
the regularised value of the sequence $(u_N)_{N \in \N}$ where
$$
u_N  :=\sum_{N > n_1 > \cdots >n_r>0} \frac{\log^{k_1}n_1 \cdots \log^{k_r} n_r}
{n_1^{a_1} \cdots n_r^{a_r}},
$$
is denoted by $\gamma_{k_1, \ldots, k_r}^{(a_1,\ldots,a_r)}$ and called
{\rm the multiple Stieltjes constant of order $(k_1, \ldots, k_r)$
at the point $(a_1,\ldots,a_r)$}.
\end{defn}

\section{Behaviour of the multiple zeta functions around $(1,\ldots,1)$}\label{boundary-1}

In this special case the multiple Stieltjes constants
$\gamma_{k_1, \ldots, k_r}^{(1,\ldots,1)}$ are simply denoted by $\gamma_{k_1,\ldots,k_r}$.

\begin{thm}\label{thm-multi-zeta-exp}
Let $r \ge 0$ be an integer. The power series
\begin{equation}\label{reg-multi-zeta-1}
\sum_{k_1,\ldots,k_r \ge 0}
\frac{(-1)^{k_1+\cdots+k_r}}{k_1! \cdots k_r!}
\gamma_{k_1, \ldots, k_r} (s_1-1)^{k_1} \cdots (s_r-1)^{k_r}
\end{equation}
converges in a neighbourhood of the point $(1, \ldots, 1)$ of $\C^r$.
It extends to a meromorphic function on $\C^r$, denoted by
$\z^\reg_{(1, \ldots, 1)} (s_1,\ldots,s_r)$ and we have the following
equality between meromorphic functions on $\C^r :$
\begin{equation}\label{multi-zeta-exp}
\z(s_1,\ldots,s_r) = \sum_{i=0}^r
\frac{\z^\reg_{(1, \ldots, 1)} (s_{i+1},\ldots,s_r)}{(s_1-1) \cdots (s_1+\cdots+s_i-i)}.
\end{equation}
\end{thm}

Note that in \eqref{multi-zeta-exp}, the term of index $0$ in the sum is
$\z^\reg_{(1, \ldots, 1)} (s_1,\ldots,s_r)$ and the term of index $r$ is
$\frac{1}{(s_1-1) \cdots (s_1+\cdots+s_r-r)}$.

\begin{exm}\rm
We have, in a neighbourhood of $(1,1)$, the following Laurent type
expansion of $\z(s_1,s_2)$ : 
\begin{align*}
\z(s_1,s_2)= & \frac{1}{(s_1-1)(s_1+s_2-2)} + \frac{1}{s_1-1} \sum_{k \ge 0}
\frac{(-1)^k \gamma_k}{k!} (s_2-1)^k \\
& + \sum_{k_1,k_2 \ge 0}\frac{(-1)^{k_1+k_2} \gamma_{k_1,k_2}}{k_1! k_2!}
(s_1-1)^{k_1} (s_2-1)^{k_2}.
\end{align*}
\end{exm}

\begin{proof}[Proof of \thmref{thm-multi-zeta-exp}]
We prove this theorem by induction on the depth $r$. When $r=0$, we just have
$\z(\varnothing)=1$ and $\z^\reg_{(\varnothing)}(\varnothing)=1$, hence the 
theorem is true. Next let $r \ge 1$. It is enough to prove that the power
series \eqref{reg-multi-zeta-1} converges in a neighbourhood of $(1,\ldots,1)$
to a function satisfying \eqref{multi-zeta-exp} in this neighbourhood.
The meromorphic continuation will then follow from the induction
hypothesis. To do this we use the following general lemma.

\begin{lem}\label{lem-conv}
Let ${\bf a}=(a_1,\ldots,a_r)$ be a point in $\C^r$ and $D$ denote a polydisc around
${\bf a}$. Let $(u_N)$ be a sequence of holomorphic
functions on $D$ which converges uniformly on $D$ to a function
$u$. Assume that for all integers $k_1,\ldots,k_r \ge 0$, the $(k_1,\ldots,k_r)$-th 
coefficient of the Taylor expansion of $u_N$ at ${\bf a}$ has the 
form $P_{k_1,\ldots,k_r}(N,\log N) + o(1)$ as $N \to \infty$,
where $P_{k_1,\ldots,k_r}$ is a polynomial in two indeterminates with complex
coefficients. Then $P_{k_1,\ldots,k_r}$ is a constant polynomial. If this
constant is $\alpha_{k_1,\ldots,k_r}$, then the function $u$, in the polydisc $D$,
is the sum for $(z_1,\ldots,z_r) \in D$ of the convergent power series
$$
\sum_{k_1,\ldots,k_r \ge 0} \alpha_{k_1,\ldots,k_r}(z_1-a_1)^{k_1}\cdots (z_r-a_r)^{k_r}.
$$
\end{lem}

\begin{proof}[Proof of \lemref{lem-conv}]
Since $(u_N)$ converges uniformly to $u$, $u$ is holomorphic in $D$ and
hence given by its Taylor expansion at ${\bf a}$. Furthermore, for any fixed
integers $k_1,\ldots,k_r \ge 0$, the sequence of $(k_1,\ldots,k_r)$-th coefficient
of the Taylor expansion of $(u_N)$ at ${\bf a}$ converges to 
$(k_1,\ldots,k_r)$ coefficient of the Taylor expansion of $u$ at that
point. Since the $(k_1,\ldots,k_r)$-th 
coefficient of the Taylor expansion of $u_N$ at ${\bf a}$ has the 
form $P_{k_1,\ldots,k_r}(N,\log N) + o(1)$ as $N \to \infty$,
$P_{k_1,\ldots,k_r}$ must be a constant polynomial. This
constant is then nothing but the $(k_1,\ldots,k_r)$-th coefficient of the Taylor
expansion of $u$ at the point ${\bf a}$. This completes the proof
of the lemma.
\end{proof}

Now we start with the following series expansion which is valid
for any integer $n_1 \ge 2$ and complex number $s_1$ :
\begin{equation}\label{first-step}
n_1^{1-s_1} - (n_1+1)^{1-s_1}= \sum_{k \ge 0} (-1)^k
\frac{(s_1-1)_{k+1}}{(k+1)!} \ n_1^{-s_1-k},
\end{equation}
where for $s \in \C$ and $k \ge 0$,
$$
(s)_k  := s (s+1) \cdots (s+k-1).
$$
For any $(s_1,\ldots,s_r) \in \C^r$ and any integer $N \ge 1$, let us define
\begin{equation}\label{zeta-initial}
\z(s_1,\ldots,s_r)_{<N} := \sum_{N > n_1 > \cdots >n_r>0} n_1^{-s_1} \cdots n_r^{-s_r}.
\end{equation}
Let $\xi_N$ denote the meromorphic function $\z(s_1,\ldots,s_r) - \z(s_1,\ldots,s_r)_{<N}$
on $\C^r$, which on $U_r$ is given by the absolutely convergent series
\begin{equation}\label{zeta-rest}
\xi_N(s_1,\ldots,s_r) =\sum_{n_1 > \cdots >n_r>0, n_1 \ge N} n_1^{-s_1} \cdots n_r^{-s_r}.
\end{equation}
When we multiply both sides of \eqref{first-step} by $n_2^{-s_2} \cdots n_r^{-s_r}$
and sum for $n_1 > \cdots >n_r>0$ with $n_1 \ge N \ge 2$ and $(s_1,\ldots,s_r) \in U_r$, we get
\begin{equation}\label{second-step}
\begin{split}
& N^{1-s_1}\z(s_2,\ldots,s_r)_{<N}
+ \xi_N(s_1+s_2-1, s_3,\ldots,s_r) \\
& =\sum_{k\ge 0}(-1)^k \frac{(s_1-1)_{k+1}}{(k+1)!} \xi_N(s_1+k,s_2,\ldots,s_r).
\end{split}
\end{equation}
The interchange of summations on the right hand side is justified as the family
$$
\left((-1)^k \frac{(s_1-1)_{k+1}}{(k+1)!}  n_1^{-s_1-k} n_2^{-s_2} \cdots n_r^{-s_r}
\right)_{n_1 > \cdots >n_r>0 \atop n_1 \ge N \ge 2; k \ge 0}
$$
is normally summable on any compact subset of $U_r$ (see \cite[Proposition 2]{MSV}).
We now prove the following general lemma.
For a real number $x$ and $(a_1,\ldots,a_r)\in \C^r$, let $\tau_x((a_1,\ldots,a_r))$
denote the point $(a_1+x,a_2,\ldots,a_r)\in \C^r$. For a set $X$ and a complex
valued bounded function $f :X \to \C$, we denote $\|f\|_X := \sup_{x \in X} |f(x)|$.

\begin{lem}\label{lem-negligible}
Let $K$ be a compact subset of $\C^r$ and $A$ be a non-negative integer.
Suppose $k_0$ is the smallest non-negative integer such that
$\tau_{k_0}(K) \subset \tau_{A}(U_r)$. Then the family
$$
\left(\|(-1)^k \frac{(s_1-1)_{k+1}}{(k+1)!}  n_1^{-s_1-k} n_2^{-s_2} \cdots n_r^{-s_r}\|_K
\right)_{n_1 > \cdots >n_r>0 \atop n_1 \ge N \ge 2; k \ge k_0}
$$
is summable and its sum is $o(N^{-A})$ as $N$ tends to $\infty$.
\end{lem}

\begin{proof}[Proof of \lemref{lem-negligible}]
We have $\tau_{k_0-A}(K) \subset U_r$. Since $K$ is compact, we can in fact find
$\epsilon > 0$ such that $\tau_{k_0-A-\epsilon}(K) \subset U_r$. Then for
$n_1 \ge N$ and $k \ge k_0$,
$$
\|n_1^{-s_1-k} n_2^{-s_2} \cdots n_r^{-s_r}\|_K
\le N^{-A-\epsilon-k+k_0} \|n_1^{-s_1-k_0+A+\epsilon} n_2^{-s_2} \cdots n_r^{-s_r}\|_K.
$$
Since $\tau_{k_0-A-\epsilon}(K) \subset U_r$, the family
$$
(n_1^{-s_1-k_0+A+\epsilon} n_2^{-s_2} \cdots n_r^{-s_r})_{n_1 > \cdots >n_r>0}
$$
is summable. On the other hand, if $M := \|s_1-1\|_K$, we have
$\left\| (-1)^k \frac{(s_1-1)_{k+1}}{(k+1)!} \right\|_K \le \frac{(M)_{k+1}}{(k+1)!}$.
Now for $N \ge 2$, the sum $\sum_{k\ge k_0}\frac{(M)_{k+1}}{(k+1)!}N^{-A-\epsilon-k+k_0}$ 
is summable and it is $o(N^{-A})$ as $N$ tends to $\infty$, since it is bounded
above by the convergent series
$N^{-A-\epsilon} \sum_{k\ge k_0}\frac{(M)_{k+1}}{(k+1)! \ 2^{k-k_0}}$.
This completes the proof of \lemref{lem-negligible}.
\end{proof}

Let $D$ be an open polydisc with center $(1,\ldots,1)$ and polyradius
$(\rho_1,\ldots,\rho_r)$ such that $\rho_1+\cdots+\rho_r<1$.
We deduce from formula \eqref{second-step} and \lemref{lem-negligible}
(for $K=\overline{D},A=0$ and $k_0=1$) that, for $N \ge 2$, the function
$$
N^{1-s_1}\z(s_2,\ldots,s_r)_{<N} + \xi_N(s_1+s_2-1, s_3,\ldots,s_r)-
(s_1-1) \xi_N(s_1,\ldots,s_r)
$$
is holomorphic in $D$, and that it converges
uniformly to $0$ as $N$ tends to $\infty$. This implies that the meromorphic function
\begin{equation}\label{limit-func-1}
u(s_1,\ldots,s_r) :=(s_1-1)\z(s_1,\ldots,s_r) - \z(s_1+s_2-1, s_3,\ldots,s_r)
\end{equation}
is holomorphic in $D$ and the sequence of holomorphic functions $(u_N)_{N \ge 2}$
defined by
\begin{equation}\label{seq-func-1}
\begin{split}
u_N(s_1,\ldots,s_r) := & N^{1-s_1}\z(s_2,\ldots,s_r)_{<N} - \z(s_1+s_2-1, s_3,\ldots,s_r)_{<N}\\
& +(s_1-1) \z(s_1,\ldots,s_r)_{<N}
\end{split}
\end{equation}
converges uniformly to $u$ on $D$.

Let
$$
\sum_{k_1,\ldots,k_r \ge 0} a_{k_1,\ldots,k_r}(N) (s_1-1)^{k_1} \cdots (s_r-1)^{k_r}
$$
be the Taylor expansion of $u_N$ at $(1,\ldots,1)$. We deduce from \thmref{thm-msc}
and \rmkref{rmk-order-3} that $a_{k_1,\ldots,k_r}(N)$ is of the form
$P_{k_1,\ldots,k_r}(\log N) + o(1)$ as $N$ tends to $\infty$, where $P_{k_1,\ldots,k_r}$ is
a polynomial in $\C[L]$.
Hence by \lemref{lem-conv}, $P_{k_1,\ldots,k_r}$ must be a
constant polynomial, which is $a_{k_1,\ldots,k_r}$, where $a_{k_1,\ldots,k_r}$
is the $(k_1,\ldots,k_r)$-th Taylor coefficient of $u$ at $(1,\ldots,1)$.

The number $a_{k_1,\ldots,k_r}$, being the constant term of $P_{k_1,\ldots,k_r}$,
can be directly read from \eqref{seq-func-1} : it is the sum of the $(k_1,\ldots,k_r)$-th
Taylor coefficients of the functions
$$
(s_1,\ldots,s_r) \mapsto \z^\reg_{(1, \ldots, 1)} (s_{2},\ldots,s_r) -
\z^\reg_{(1, \ldots, 1)} (s_{1}+s_2-1,s_3,\ldots,s_r)
$$
and of the formal power series $(s_1-1)v$, where $v$ is the formal power series \eqref{reg-multi-zeta-1}.
This implies that the formal power series $(s_1-1)v$ converges on $D$. Hence $v$ converges
on $D$ and if $\z^\reg_{(1, \ldots, 1)} (s_{1},\ldots,s_r)$ denotes its sum on $D$, then
the function
$$
(s_1-1)\z(s_1,\ldots,s_r) - \z(s_1+s_2-1, s_3,\ldots,s_r)
$$
is equal to
$$
\z^\reg_{(1, \ldots, 1)} (s_{2},\ldots,s_r) -
\z^\reg_{(1, \ldots, 1)} (s_{1}+s_2-1,s_3,\ldots,s_r)
+ (s_1-1)\z^\reg_{(1, \ldots, 1)} (s_{1},\ldots,s_r)
$$
on $D$. By the induction hypothesis we further have the following equality of
meromorphic functions
\begin{align*}
& \z(s_1+s_2-1,s_3,\ldots,s_r) - \z^\reg_{(1, \ldots, 1)} (s_{1}+s_2-1,s_3,\ldots,s_r)\\
& = \sum_{i=2}^r
\frac{\z^\reg_{(1, \ldots, 1)} (s_{i+1},\ldots,s_r)}{(s_1+s_2-2) \cdots (s_1+\cdots+s_i-i)}.
\end{align*}
Thus we get
$$
\z(s_1,\ldots,s_r) = \sum_{i=0}^r
\frac{\z^\reg_{(1, \ldots, 1)} (s_{i+1},\ldots,s_r)}{(s_1-1) \cdots (s_1+\cdots+s_i-i)}
$$
on $D$ and this completes the proof of \thmref{thm-multi-zeta-exp}.
\end{proof}

\begin{rmk}\label{rmk-inversion}\rm
Formula \eqref{multi-zeta-exp} allows us to express the multiple zeta functions
in terms of their regularised counterparts at $(1, \ldots, 1)$. Conversely we can also
deduce from \eqref{multi-zeta-exp} an expression of these regularised multiple zeta
functions  in terms of the multiple zeta functions themselves as follows :
\begin{equation}\label{converse-multi-zeta-exp}
\z^\reg_{(1, \ldots, 1)} (s_1,\ldots,s_r)
= \sum_{i=0}^r
\frac{(-1)^i \z(s_{i+1},\ldots,s_r)}{(s_i-1)(s_i+s_{i-1}-2) \cdots (s_i+\cdots+s_1-i)}.
\end{equation}
A proof together with more details about this inversion process, will be given
in Section \ref{gen-boundary}.
\end{rmk}

\begin{rmk}\label{rmk-reg-poles}\rm
We have already noticed that $\z^\reg_{(1)} (s)$ is an entire function. It is not true
any more for $\z^\reg_{(1,1)} (s_1,s_2)$. Indeed, we have
$$
\z(s_1,s_2)= \frac{1}{(s_1-1)(s_1+s_2-2)} + \frac{\z^\reg_{(1)} (s_2)}{s_1-1}
+\z^\reg_{(1,1)} (s_1,s_2)
$$
and $\z(s_1,s_2)$ has simple poles along the hyperplanes with equation $s_1=1$ and
$s_1+s_2=2,1,0,-2,-4,\ldots$. Since $\z^\reg_{(1,1)} (s_1,s_2)$ is holomorphic
around $(1,1)$, it has no polar singularities along the hyperplanes with equation
$s_1=1$ and $s_1+s_2=2$. But it has simple poles along each of the hyperplanes
with equation $s_1+s_2=1,0,-2,-4,\ldots$.

For $r\ge 3$, we encounter a new feature. When $r=3$ we have
\begin{align*}
\z(s_1,s_2,s_3) = & \frac{1}{(s_1-1)(s_1+s_2-2)(s_1+s_2+s_3-3)}
+ \frac{\z^\reg_{(1)} (s_3)}{(s_1-1)(s_1+s_2-2)}\\
& +\frac{\z^\reg_{(1,1)} (s_2,s_3)}{s_1-1}+\z^\reg_{(1,1,1)} (s_1,s_2,s_3)
\end{align*}
and since the hyperplanes with equation $s_2+s_3=1,0,-2,-4,\ldots$ are polar hyperplanes of
$\z^\reg_{(1,1)} (s_2,s_3)$ and not of $\z(s_1,s_2,s_3)$, they are polar hyperplanes of
$\z^\reg_{(1,1,1)} (s_1,s_2,s_3)$.
That the meromorphic function $\z^\reg_{(a_1, \ldots, a_r)} (s_1,\ldots,s_r)$
can have polar hyperplanes other than those of $\z(s_1,\ldots,s_r)$, when $r\ge 3$, will
be more evident from \eqref{reg-exp} in Section \ref{gen-boundary}.
\end{rmk}

Formula \eqref{multi-zeta-exp} is a generalisation of formula \eqref{zeta-exp}
for any depth $r \ge 1$. We therefore consider it as a Laurent type
expansion of $\z(s_1,\ldots,s_r)$ around the point $(1, \ldots, 1)$. For such
an expansion, we have the following unicity property.

\begin{prop}\label{prop-unicity}
Let $\rho$ be a positive real number and $D_r(\rho)$ denote the open polydisc
in $\C^r$ with center at the point $(1, \ldots, 1)$ and polyradius $(\rho, \ldots, \rho)$.
If in such a polydisc we have
\begin{equation}\label{unicity}
\sum_{i=0}^r \frac{h_i(s_{i+1},\ldots,s_r)}{(s_1-1) \cdots (s_1+\cdots+s_i-i)}=0,
\end{equation}
where $h_i$ is holomorphic in $D_{r-i}(\rho)$, then all $h_i$ are $0$.
\end{prop}

\begin{proof}
We argue by contradiction. Let $j$ denote the largest natural number $\le r$
such that $h_j\neq 0$. We then multiply \eqref{unicity} by $(s_1-1) \cdots (s_1+\cdots+s_j-j)$
and then restrict this equality to a point of the form $(1,\ldots,1,s_{j+1},\ldots,s_r)$,
with $(s_{j+1},\ldots,s_r) \in D_{r-j}(\rho)$. We get that $h_j(s_{j+1},\ldots,s_r)=0$,
which is a contradiction.
\end{proof}

\section{Generalisation to some integral points in $\overline{U_r}$}\label{pos-boundary}
Here we consider any point $(a_1,\ldots,a_r) \in \Z^r$
which is of the form $(1,\ldots,1, a_{l+1},\ldots,a_r)$ for some $0 \le l \le r$
and $(a_{l+1},\ldots,a_r) \in U_{r-l}$. 
In fact, around such a point the Laurent type expansion of $\z(s_1,\ldots,s_r)$ is
similar to \eqref{multi-zeta-exp}. Note that the polar hyperplanes
of $\z(s_1,\ldots,s_r)$ passing through the point $(a_1,\ldots,a_r)$
are the ones given by the following equations :
$$
s_1=1, s_1+s_2=2, \ldots, s_1+\cdots+s_l=l.
$$

\begin{thm}\label{thm-gen-multi-zeta-exp}
Let $r \ge 0$ be an integer and $(a_1,\ldots,a_r)$ be as above. The power series \eqref{reg-multi-zeta}
$$
\sum_{k_1,\ldots,k_r \ge 0}
\frac{(-1)^{k_1+\cdots+k_r}}{k_1! \cdots k_r!}
\gamma_{k_1, \ldots, k_r}^{(a_1,\ldots,a_r)} (s_1-a_1)^{k_1} \cdots (s_r-a_r)^{k_r}
$$
converges in a neighbourhood of the point $(a_1,\ldots,a_r)$ of $\C^r$.
It extends to a meromorphic function on $\C^r$, denoted by
$\z^\reg_{(a_1,\ldots,a_r)} (s_1,\ldots,s_r)$ and we have the following
equality between meromorphic functions on $\C^r :$
\begin{equation}\label{gen-multi-zeta-exp}
\z(s_1,\ldots,s_r) = \sum_{i=0}^l
\frac{\z^\reg_{(a_{i+1},\ldots,a_r)} (s_{i+1},\ldots,s_r)}{(s_1-1) \cdots (s_1+\cdots+s_i-i)}.
\end{equation}
\end{thm}

\begin{exm}\rm
We have, in a neighbourhood of $(1,2)$  :
$$
\z(s_1,s_2)=\frac{1}{s_1-1} \sum_{k \ge 0} \frac{(-1)^k \gamma_k^{(2)}}{k!} (s_2-2)^k
+ \sum_{k_1,k_2 \ge 0} \frac{(-1)^{k_1+k_2} \gamma_{k_1,k_2}^{(1,2)}}{k_1! k_2!}
(s_1-1)^{k_1} (s_2-2)^{k_2},
$$
where $(-1)^k \gamma_k^{(2)}$ is just $D^{k}\z(2)$.
\end{exm}

\begin{proof}[Proof of \thmref{thm-gen-multi-zeta-exp}]
The proof is by induction on $l$ (for arbitrary $r$). When $l=0$, \eqref{gen-multi-zeta-exp}
is nothing but \eqref{multi-zeta-conv}. The induction then carries out mutatis-mutandis
with $(1,\ldots,1)$ in $\C^r$ replaced by $(1,\ldots, 1, a_{l+1},\ldots,a_r)$
in the proof of \thmref{thm-multi-zeta-exp}.
\end{proof}

\section{A combinatorial formula}\label{comb-form}

We need a general combinatorial formula (see \eqref{comb-form-1} below) satisfied by
the multiple zeta functions, to explain their local behaviour at any integral point in the
closure of the domain of convergence.

For $r\ge 0$, the multiple zeta-star function of depth $r$
is defined on $U_r$ by the series expression
\begin{equation}\label{zeta-star}
\z^\star(s_1,\ldots,s_r) := \sum_{n_1 \ge \cdots \ge n_r \ge 1} n_1^{-s_1} \cdots n_r^{-s_r},
\end{equation}
which converges normally on any compact subset of $U_r$. In particular,
the multiple zeta-star function of depth $0$ is defined by $\z^\star(\varnothing) :=1$.
The multiple zeta-star function of depth $r$ has a meromorphic extension to $\C^r$,
as can be seen by expressing it in terms of the multiple zeta functions of depth $\le r$.

Recall that in \eqref{zeta-initial} we have defined a holomorphic function on $\C^r$ by
$$
\z(s_1,\ldots,s_r)_{<N} := \sum_{N > n_1 > \cdots >n_r>0} n_1^{-s_1} \cdots n_r^{-s_r}
$$
for any integer $N \ge 1$. Similarly for any integer $N \ge 1$,
we define on $\C^r$ a holomorphic function by
\begin{equation}\label{zeta-star-initial}
\z^\star(s_1,\ldots,s_r)_{\le N}  := \sum_{N \ge n_1 \ge \cdots \ge n_r \ge 1}
n_1^{-s_1} \cdots n_r^{-s_r}.
\end{equation}
For $(s_1,\ldots,s_r) \in U_r$, we also consider the tails
\begin{equation}\label{zeta-tail}
\z(s_1,\ldots,s_r)_{>N}   := \sum_{n_1 > \cdots >n_r>N} n_1^{-s_1} \cdots n_r^{-s_r},
\end{equation}
and
\begin{equation}\label{zeta-star-tail}
\z^\star(s_1,\ldots,s_r)_{\ge N}  := \sum_{n_1 \ge \cdots \ge n_r \ge N}
n_1^{-s_1} \cdots n_r^{-s_r}.
\end{equation}
In depth $0$, we use the conventions that
$\z(\varnothing)_{< N}=\z^\star(\varnothing)_{\le N}=
\z(\varnothing)_{> N}=\z^\star(\varnothing)_{\ge N}  :=1$.
The infinite sums in \eqref{zeta-tail} and \eqref{zeta-star-tail} are in fact normally
convergent on any compact subset of $U_r$ and hence define holomorphic functions there.
From \cite[Remark 1]{MSV}, it follows that $\z(s_1,\ldots,s_r)_{>N}$ has a
meromorphic extension to $\C^r$. The same is then true for $\z^\star(s_1,\ldots,s_r)_{\ge N}$.
With these notations in place we prove the following combinatorial formula.

\begin{thm}\label{thm-comb-form-1}
For each integer $N \ge 1$, we have the following equality between
meromorphic functions on $\C^r  :$
\begin{equation}\label{comb-form-1}
\z(s_1,\ldots,s_r)_{< N}= \sum_{i=0}^r
(-1)^i \z^\star (s_i,\ldots,s_1)_{\ge N} \z(s_{i+1},\ldots,s_r).
\end{equation}
\end{thm}

\begin{proof}
It is enough to prove \eqref{comb-form-1} as an equality between
holomorphic functions in the open set
$$
V_r  :=\{(s_1,\ldots,s_r) \in \C^r   : \Re(s_i)>1 \text{ for } 1\le i \le r\}.
$$
We have, for $0 \le i \le r$,
\begin{equation}\label{comb-form-1-RHS-term}
\z^\star (s_i,\ldots,s_1)_{\ge N} \z(s_{i+1},\ldots,s_r)
= \sum_{(n_1, \ldots,n_r) \in A_i} n_1^{-s_1} \cdots n_r^{-s_r},
\end{equation}
where
\begin{equation*}\label{set-A-i}
A_i  :=\{(n_1, \ldots,n_r) \in \N^r   : n_i \ge \cdots \ge n_1 \ge N, n_{i+1} > \cdots >n_r>0\}.
\end{equation*}
Let
\begin{equation*}\label{set-B-0}
B_0  :=\{(n_1, \ldots,n_r) \in \N^r   : N > n_{1} > \cdots >n_r>0\}
\end{equation*}
and for $0 \le i \le r$,
\begin{equation*}\label{set-B-i}
B_i  :=\{(n_1, \ldots,n_r) \in \N^r   : n_i \ge \cdots \ge n_1 \ge N, n_{i} > \cdots >n_r>0\}.
\end{equation*}
Then $A_i$ is the disjoint union of $B_i$ and $B_{i+1}$ for $0 \le i <r$ and is equal to $B_r$
for $i=r$. Now by \eqref{comb-form-1-RHS-term}, the right hand side of
\eqref{comb-form-1} is
$$
\sum_{i=0}^r (-1)^i \sum_{(n_1, \ldots,n_r) \in A_i} n_1^{-s_1} \cdots n_r^{-s_r}
= \sum_{(n_1, \ldots,n_r) \in B_0} n_1^{-s_1} \cdots n_r^{-s_r} = \z(s_1,\ldots,s_r)_{< N}.
$$
This completes the proof of \thmref{thm-comb-form-1}.
\end{proof}

Similarly one can also obtain the formula
\begin{equation}\label{comb-form-2}
\z^\star(s_1,\ldots,s_r)_{\le N}= \sum_{i=0}^r
(-1)^i \z(s_i,\ldots,s_1)_{> N} \z^\star(s_{i+1},\ldots,s_r).
\end{equation}
By taking $N=1$ in \eqref{comb-form-1}, we recover the well known formula
\begin{equation}\label{comb-form-cor}
\sum_{i=0}^r (-1)^i \z^\star(s_i,\ldots,s_1) \z(s_{i+1},\ldots,s_r)=0.
\end{equation}

\section{The case of general integral points in $\overline{U_r}$}\label{gen-boundary}

For a general point $(a_1,\ldots,a_r) \in \Z^r \cap \overline{U_r}$, the Laurent
type expansion  of $\z(s_1,\ldots,s_r)$ around the point
$(a_1,\ldots,a_r)$, does not have a simple form as \eqref{multi-zeta-exp} and
\eqref{gen-multi-zeta-exp} in general (see Example 4 below). But in that
case we are able to give a different but elegant expression, from which
\eqref{multi-zeta-exp} and \eqref{gen-multi-zeta-exp} can be deduced as special cases.

\subsection{Statement of the theorem}

\begin{thm}\label{thm-reg-exp}
Let $r \ge 0$ be an integer and $(a_1,\ldots,a_r) \in \overline{U_r} \cap \Z^r$.
The power series \eqref{reg-multi-zeta}
$$
\sum_{k_1,\ldots,k_r \ge 0}
\frac{(-1)^{k_1+\cdots+k_r}}{k_1! \cdots k_r!}
\gamma_{k_1, \ldots, k_r}^{(a_1,\ldots,a_r)}
(s_1-a_1)^{k_1} \cdots (s_r-a_r)^{k_r}
$$
converges in a neighbourhood of $(a_1, \ldots, a_r)$
and extends to a meromorphic function in the whole of $\C^r$, denoted by
$\z^\reg_{(a_1, \ldots, a_r)} (s_1,\ldots,s_r)$.
Then we have the following equality
\begin{equation}\label{reg-exp}
\z^\reg_{(a_1, \ldots, a_r)} (s_1,\ldots,s_r)
= \sum_{i \in I}
\frac{(-1)^i \z(s_{i+1},\ldots,s_r)}{(s_i-1)(s_i+s_{i-1}-2) \cdots (s_i+\cdots+s_1-i)}
\end{equation}
between meromorphic functions on $\C^r$, where 
$I=I(a_1,\ldots,a_r)$ denotes the set of indices $i$
such that $0 \le i \le r$ and $a_1+\cdots+a_i=i$.
\end{thm}

\begin{exm} \rm
When $(a_1,\ldots,a_r)=(1,\ldots,1)$, \eqref{reg-exp} is nothing but \eqref{converse-multi-zeta-exp}.
\end{exm}

\begin{exm} \rm
When $r=2$ and $(a_1,a_2)=(2,0)$, we have
$$
\z^\reg_{(2, 0)} (s_1,s_2) = \z(s_1,s_2) + \frac{1}{(s_2-1) (s_1+s_2-2)},
$$
which implies that
$$
\z(s_1,s_2) = - \frac{1}{(s_2-1) (s_1+s_2-2)} + \sum_{k_1,k_2 \ge 0}
\frac{(-1)^{k_1+k_2}}{k_1! k_2!}
\gamma_{k_1, k_2}^{(2,0)}(s_1-2)^{k_1} s_2^{k_2}
$$
in a neighbourhood of $(2,0)$.
Note that the rational function $\frac{1}{(s_2-1) (s_1+s_2-2)}$ has a pole
along the line given by the equation $s_2=1$, whereas $\z(s_1,s_2)$
has no such pole. Hence the meromorphic function $\z^\reg_{(2, 0)} (s_1,s_2)$
(which is holomorphic around (2,0)) has a pole along this line.
\end{exm}

\subsection{Proof of \thmref{thm-reg-exp}}

For our proof we use the combinatorial formula \eqref{comb-form-1} and therefore
we need to estimate $\z^\star(s_r,\ldots,s_1)_{\ge N}$ around a point
$(a_1,\ldots,a_r) \in \overline{U_r}$.

Note that when $(s_r,\ldots,s_1) \in U_r$, we can deduce a translation
formula satisfied by $\z^\star(s_r,\ldots,s_1)_{\ge N}$, starting with \eqref{first-step}.
This reads as follows  : for $r=1$,
\begin{equation}\label{second-step-star-tail-r=1}
N^{1-s_1} =\sum_{k\ge 0}(-1)^k \frac{(s_1-1)_{k+1}}{(k+1)!} \z^\star(s_1+k)_{\ge N},
\end{equation}
and for $r >1$,
\begin{equation}\label{second-step-star-tail}
\z^\star(s_r+s_{r-1}-1, s_{r-2},\ldots,s_1)_{\ge N}
=\sum_{k\ge 0}(-1)^k \frac{(s_r-1)_{k+1}}{(k+1)!} \z^\star(s_r+k,s_{r-1},\ldots,s_1)_{\ge N}.
\end{equation}
Formulas \eqref{second-step-star-tail-r=1}, \eqref{second-step-star-tail} can be extended
to whole of $\C^r$ as equalities between meromorphic functions. In fact, for any given point
$(a_1,\ldots,a_r) \in \C^r$, there exists a natural number $k_0$ such that
$(a_r+k_0, a_{r-1},\ldots,a_1) \in U_r$ and hence there exists a polydisc $D$
with center $(a_1,\ldots,a_r)$ such that
for all $k \ge k_0$, $\z^\star(s_r+k,s_{r-1},\ldots,s_1)_{\ge N}$
is holomorphic in $D$. Moreover, \lemref{lem-negligible} shows that the sum
$$
\sum_{k\ge k_0} \left\| (-1)^k \frac{(s_r-1)_{k+1}}{(k+1)!}
\z^\star(s_r+k,s_{r-1},\ldots,s_1)_{\ge N}\right\|_D
$$
exists and it is $o(1)$ as $N$ tends to $\infty$. Using this we now prove the following lemma,
which constitutes an important step in our proof of \thmref{thm-reg-exp}.

\begin{lem}\label{lem-star-tail-estimate}
Let $r \ge 0$ be an integer and $(a_1,\ldots,a_r) \in \C^r$ such that $a_1+\cdots+a_r \ge r$.
There exists a non-zero polynomial $P(s_1,\ldots,s_r)$ which is a multiple of
$(s_r-1)\cdots (s_r+\cdots+s_1-r)$ and a polydisc $D$
with center $(a_1,\ldots,a_r)$ such that \\
a) for any $N \ge 2$, the function $u_N$ is holomorphic in $D$, where
$$
u_N :=P(s_1,\ldots,s_r) \z^\star(s_r,\ldots,s_1)_{\ge N};
$$
b) if $a_1+\cdots+a_r >r$, then as $N$ tends to $\infty$,
$$
\|u_N\|_D=o(1);
$$
c) if $a_1+\cdots+a_r =r$, then as $N$ tends to $\infty$,
$$
\left\|u_N - \frac{P(s_1,\ldots,s_r)  N^{r-s_r-\cdots-s_1}}
{(s_r-1)\cdots (s_r+\cdots+s_1-r)}\right\|_D=o(1).
$$
\end{lem}

\begin{proof}
We prove this by double induction, firstly on the depth $r$
and then on the smallest integer $k_0 \ge 0$ such that
$(a_r+k_0,a_{r-1},\ldots,a_1) \in U_r$. The case when $r=0$
is easily done as in this case we are in the case b)
and therefore we have the desired result with $P(\varnothing) :=1$.

Now for $r\ge 1$, from our discussion above we know that the sequence of meromorphic
functions $(v_N)_{N \ge 2}$ with
$$
v_N :=\z^\star(s_r+s_{r-1}-1, s_{r-2},\ldots,s_1)_{\ge N}
-\sum_{k= 0}^{k_0-1}(-1)^k \frac{(s_r-1)_{k+1}}{(k+1)!} \z^\star(s_r+k,s_{r-1},\ldots,s_1)_{\ge N}
$$
is holomorphic in a neighbourhood $D$ of $(a_1,\ldots,a_r)$ and converges uniformly to $0$
on $D$, as $N$ tends to $\infty$. It is therefore enough to prove our result for the functions
$\z^\star(s_r+s_{r-1}-1, s_{r-2},\ldots,s_1)_{\ge N}$ and $\z^\star(s_r+k,s_{r-1},\ldots,s_1)_{\ge N}$
for each $k=1,\ldots,k_0-1$. Since $a_1+\cdots+a_r+k > r$ for each $k=1,\ldots,k_0-1$,
by the induction hypothesis for depth $r$ and $k_0-k$, we have a desired polynomial
$Q(s_1,\ldots,s_r)$ such that
$$
\|Q(s_1,\ldots,s_r) \z^\star(s_r+k,\ldots,s_1)_{\ge N}\|_D=o(1)
$$
as $N \to \infty$, for each $k=1,\ldots,k_0-1$. On the other hand, the induction hypothesis
for depth $r-1$ applies to the function $\z^\star(s_r+s_{r-1}-1, s_{r-2},\ldots,s_1)_{\ge N}$
around the point $(a_r+a_{r-1}-1, a_{r-2},\ldots,a_1)$ and yields a polynomial
$R(s_1,\ldots,s_{r-2},s_{r-1}+s_r)$ with the desired property. We choose
$P(s_1,\ldots,s_r)$ to be $(s_1-1)Q(s_1,\ldots,s_r) R(s_1,\ldots,s_{r-2},s_{r-1}+s_r)$
to complete the proof of \lemref{lem-star-tail-estimate}.
\end{proof}

\begin{proof}[Proof of \thmref{thm-reg-exp}]
First note that $(a_1,\ldots,a_r) \in \overline{U_r}$ implies
$(a_1,\ldots,a_i) \in \overline{U_i}$ for $i=1,\ldots,r$.
Hence applying \lemref{lem-star-tail-estimate} to each of
$\z^\star(s_i,\ldots,s_1)_{\ge N}$ for $i=1,\ldots,r$, we can find a common
polynomial $P_1(s_1,\ldots,s_r)$ with the desired property
in an open polydisc $D$ of $(a_1,\ldots,a_r)$. Further, we can find a common
polynomial $P_2(s_1,\ldots,s_r)$ such that for each $i=0,\ldots,r-1$, the function
$P_2(s_1,\ldots,s_r) \z(s_{i+1},\ldots,s_r)$ is holomorphic on $D$. Therefore combining
\lemref{lem-star-tail-estimate} with \eqref{comb-form-1}, we get that for
$P(s_1,\ldots,s_r)=P_1(s_1,\ldots,s_r)P_2(s_1,\ldots,s_r)$,
the sequence of holomorphic functions $(u_N)_{N \ge 2}$ with
$$
u_N :=P(s_1,\ldots,s_r) \left(\z(s_1,\ldots,s_r)_{< N} - \sum_{i \in I}
(-1)^i \frac{N^{i-s_i-\cdots-s_1}\z(s_{i+1},\ldots,s_r)}
{(s_i-1)(s_i+s_{i-1}-2) \cdots (s_i+\cdots+s_1-i)} \right)
$$
converges uniformly to $0$ as $N$ tends to $\infty$. Note that for $i \in I$,
the Taylor expansion of $N^{i-s_i-\cdots-s_1}$ at $(a_1,\ldots,a_i)$ is
$$
\sum_{k_1,\ldots,k_i \ge 0} \frac{(-\log N)^{k_1+\cdots+k_i}}{k_1!\cdots k_i!}
(s_1-a_1)^{k_1} \cdots (s_i-a_i)^{k_i}.
$$
Hence \lemref{lem-conv} applies to the sequence of functions $(u_N)_{N \ge 2}$
and we get that the formal power series $P(s_1,\ldots,s_r)v$,
where $v$ is the formal power series \eqref{reg-multi-zeta}, converges on $D$ and
is equal to
$$
P(s_1,\ldots,s_r) \sum_{i \in I} \frac{(-1)^i \z(s_{i+1},\ldots,s_r)}
{(s_i-1)(s_i+s_{i-1}-2) \cdots (s_i+\cdots+s_1-i)}
$$
on $D$. Hence $v$ converges on $D$ and since $\z^\reg_{(a_1, \ldots, a_r)} (s_{1},\ldots,s_r)$
denotes its sum on $D$, we therefore have \eqref{reg-exp} on $D$.
\end{proof}

\subsection{An inversion process}\label{sec-inv}
In this section, we discuss inversion of an upper triangular matrix of rational
functions which we have encountered in the previous section. This inversion process
is going to play an important role in deriving a Laurent type expansion
of $\z(s_1,\ldots,s_r)$ around any integral point in $\overline{U_r}$ (see Section \ref{sec-exp-boundary}).

For that let $r \ge 0$ be an integer and $I$ be a subset of $\{0,1,\ldots,r\}$.
Let ${\bf A}=(a_{i,j})_{i,j \in I}$ be the upper triangular matrix
of type $I \times I$, with entries from the field of rational functions
$\Q(X_1,\ldots,X_r)$, defined as follows :
\begin{equation}\label{mat-A-inv}
a_{i,j} :=
\begin{cases}
0 & \text{for } i>j,\\
\prod_{m=i+1}^j (X_m+ \cdots + X_j)^{-1} & \text{for } i\le j.
\end{cases}
\end{equation}
Note that for any positive real numbers $x_1,\ldots,x_r$ and $i \le j$, we have
\begin{equation}\label{a-ij}
a_{i,j}(x_1,\ldots,x_r) = \int_{A_{i,j}}
t_{i+1}^{x_{i+1}-1} \cdots t_{j}^{x_{j}-1} dt_{i+1} \cdots dt_{j},
\end{equation}
where for $i \le j$, $A_{i,j} := \{(t_{i+1},\ldots,t_{j}) \in [0,1]^{j-i}  :
t_{i+1}>\cdots>t_{j}\}$. Similarly for $i \le j$, let
$$
B_{i,j} :=\{(t_{i+1},\ldots,t_{j}) \in [0,1]^{j-i}  : \text{for }
i < m < j, t_m > t_{m+1} \text{ if } m \notin I
\text{ and } t_m < t_{m+1} \text{ if } m \in I\}.
$$
There exists a unique rational function $b_{i,j}$ in $\Q(X_1,\ldots,X_r)$
such that for any positive real numbers $x_1,\ldots,x_r$, we have
\begin{equation}\label{b-ij}
b_{i,j}(x_1,\ldots,x_r) = \int_{B_{i,j}}
t_{i+1}^{x_{i+1}-1} \cdots t_{j}^{x_{j}-1} dt_{i+1} \cdots dt_{j}.
\end{equation}

\begin{prop}\label{prop-inverse}
The matrix ${\bf A}^{-1}$ is given by
$(a'_{i,j})_{i,j \in I}$, where
\begin{equation}\label{mat-A-inverse}
a'_{i,j}=
\begin{cases}
0 & \text{for } i>j,\\
(-1)^{|I \cap\{i+1,\ldots,j\}|} b_{i,j} & \text{for } i\le j.
\end{cases}
\end{equation}
\end{prop}

\begin{proof}
Since our matrices are upper triangular, we only have to prove that,
for $i \le k$ in $I$,
$$
\sum_{i \le j \le k \atop j \in I} a_{i,j} a'_{j,k}
= \begin{cases}
  1 & \text{if } i=k,\\
  0 & \text{if } i \neq k.
  \end{cases}
$$
In order to prove this equality between rational functions, it is enough
to prove it after replacing the indeterminates $X_1,\ldots,X_r$ by any
positive real numbers $x_1,\ldots,x_r$. We can therefore prove
it by using the respective integral representations of
$a_{i,j}(x_1,\ldots,x_r)$ and $a'_{j,k}(x_1,\ldots,x_r)$. 
Note that for all $i \in I$, $a_{i,i}=a'_{i,i}=1$.
This completes the proof if $i=k$. Now suppose $i \neq k$.
In this case
$$
\sum_{i \le j \le k \atop j \in I} (a_{i,j} a'_{j,k})(x_1,\ldots,x_r)
=\sum_{i \le j \le k \atop j \in I} (-1)^{|I \cap\{j+1,\ldots,k\}|}
\int_{A_{i,j}\times B_{j,k}} t_{i+1}^{x_{i+1}-1} \cdots t_{k}^{x_{k}-1}
dt_{i+1} \cdots dt_{k}.
$$
Define
$$
I_0=I_0(i,k) :=\{j \in I  : i \le j \le k \text{ and }
|I \cap\{j+1,\ldots,k\}| \text{ is even}\}.
$$
Similarly let
$$
I_1=I_1(i,k) :=\{j \in I  : i \le j \le k \text{ and }
|I \cap\{j+1,\ldots,k\}| \text{ is odd}\}.
$$
Since $i \neq k$, $I_0,I_1$ are non-empty. Note that the elements of
$I_0$ and $I_1$ are interlacing as elements of $I$, i.e. for any
$\epsilon \in \{0,1\}$, between any two consecutive elements
$j_1,j_2 \in I_\epsilon$, there is a unique element $j_3 \in I_{1-\epsilon}$
such that $j_1 < j_3 < j_2$. Thus if $j, j'$ are distinct elements of
some $I_\epsilon$ (with $\epsilon \in \{0,1\}$), the sets
$A_{i,j}\times B_{j,k},A_{i,j'}\times B_{j',k}$ are disjoint. Hence
\begin{align*}
\sum_{i \le j \le k \atop j \in I} (a_{i,j} a'_{j,k})(x_1,\ldots,x_r)
= & \int_{\sqcup_{j \in I_0} (A_{i,j}\times B_{j,k})}
t_{i+1}^{x_{i+1}-1} \cdots t_{k}^{x_{k}-1} dt_{i+1} \cdots dt_{k}\\
& - \int_{\sqcup_{j \in I_1} (A_{i,j}\times B_{j,k})}
t_{i+1}^{x_{i+1}-1} \cdots t_{k}^{x_{k}-1} dt_{i+1} \cdots dt_{k}.
\end{align*}

Now we claim that, modulo the sets of measure zero,
$\sqcup_{j \in I_0} (A_{i,j}\times B_{j,k})
\equiv \sqcup_{j \in I_1} (A_{i,j}\times B_{j,k}),$
i.e. the symmetric difference between these two sets is a set of measure zero.
Our claim now follows from the following lemma and this will complete the
proof of \propref{prop-inverse}.
\end{proof}

\begin{lem}\label{lem-set-measure}
Let ${\bf t}=(t_{i+1},\ldots,t_{k})$ be an element of $[0,1]^{k-i}$
such that its coordinates are pairwise distinct. Then ${\bf t}
\in \sqcup_{j \in I_0} (A_{i,j}\times B_{j,k})$ if and only if
${\bf t} \in \sqcup_{j \in I_1} (A_{i,j}\times B_{j,k})$.
\end{lem}

\begin{proof}[Proof of \lemref{lem-set-measure}]
Let $j \in I$ be such that $i \le j \le k$. If $j \ne k$, we denote $j^{+}$
to be the element in $I$ next to $j$. If $j \ne i$, we denote $j^{-}$
to be the element in $I$ preceding $j$. Note that when $j \in I_{\epsilon}$,
we have $j^{-},j^{+} \in I_{1-\epsilon}$. 
Now let ${\bf t} \in A_{i,j}\times B_{j,k}$.
If $j=i$, then ${\bf t} \in A_{i,j^{+}}\times B_{j^{+},k}$. If $j=k$,
then ${\bf t} \in A_{i,j^{-}}\times B_{j^{-},k}$. If $i<j<k$, then
${\bf t}$ belongs to $A_{i,j^{+}}\times B_{j^{+},k}$ when $t_j>t_{j+1}$,
and to $A_{i,j^{-}}\times B_{j^{-},k}$ when $t_j<t_{j+1}$.
This completes the proof of \lemref{lem-set-measure}.
\end{proof}

\subsection{Laurent type expansion for general integral points in $\overline{U_r}$}\label{sec-exp-boundary}
We recall that for a general point $(a_1,\ldots,a_r) \in \Z^r \cap \overline{U_r}$,
$I=I(a_1,\ldots,a_r)$ denotes the set of indices $i$ such that $0 \le i \le r$ and
$a_1+\cdots+a_i=i$. Note that the points of the form $(a_{i+1},\ldots,a_r)$ belong
to $\Z^{r-i} \cap \overline{U_{r-i}}$ for $i \in I$. Writing \thmref{thm-reg-exp}
for these points, we get a triangular system expressing the regularised
multiple zeta functions around them, in terms of the multiple
zeta functions themselves. This triangular system has to be inverted
to get a Laurent type expansion of $\z(s_1,\ldots,s_r)$
around $(a_1,\ldots,a_r)$.

\begin{thm}\label{thm-reg-exp-inverse}
Under the hypotheses of \thmref{thm-reg-exp}, we have
$$
\z(s_1,\ldots,s_r) = \sum_{i \in I} (-1)^{i-|I_i|} f_{i}(s_1-1,\ldots,s_i-1)
\z^\reg_{(a_{i+1}, \ldots, a_r)} (s_{i+1},\ldots,s_r),
$$
where for $i \in I$, $I_i=I \cap \{1,\ldots,i\}$ and
$f_{i}$ is the rational function is $\Q(X_1,\ldots,X_i)$
defined by the following property $:$ for any positive 
real numbers $x_1,\ldots,x_i$,
$$
f_{i}(x_1,\ldots,x_i)  = \int_{\Delta_{i}}
t_{1}^{x_{1}-1} \cdots t_{i}^{x_{i}-1} dt_{1} \cdots dt_{i},
$$
where $\Delta_{i} :=\{(t_{1},\ldots,t_{i}) \in [0,1]^{i}  :\text{for }
0 < j < i, t_j > t_{j+1} \text{ if } j \notin I
\text{ and } t_j < t_{j+1} \text{ if } j \in I\}$.
\end{thm}

Note that in the above theorem $I_0=\varnothing$ and $f_0=1$.

\begin{proof}[Proof of \thmref{thm-reg-exp-inverse}]
For the proof we rewrite \eqref{reg-exp} as
\begin{equation}\label{reg-exp-rewrite}
(-1)^r \z^\reg_{(a_1, \ldots, a_r)} (s_1,\ldots,s_r)
= \sum_{i \in I}
\frac{(-1)^{r-i} \z(s_{i+1},\ldots,s_r)}{(s_i-1)(s_i+s_{i-1}-2) \cdots (s_i+\cdots+s_1-i)}.
\end{equation}
Now these identities for all points of the form
$(a_{i+1},\ldots,a_r)$ with $i \in I$ can be written as the
matrix identity
\begin{equation}\label{reg-exp-mat}
{\bf V}^\reg={\bf A}(s_1-1,\ldots,s_r-1){\bf V},
\end{equation}
where ${\bf A}$ is the upper triangular matrix of type $I \times I$
with coefficients in $\Q(X_1,\ldots,X_r)$ defined by \eqref{mat-A-inv}
in Section \ref{sec-inv}, ${\bf V}^\reg$ and ${\bf V}$ are the column matrices where
entries are indexed by $I$, the entry of index $i$ being
$(-1)^{r-i}\z^\reg_{(a_{i+1}, \ldots, a_r)} (s_{i+1},\ldots,s_r)$ and
$(-1)^{r-i}\z (s_{i+1},\ldots,s_r)$ respectively.
We have seen in Section \ref{sec-inv} that the matrix ${\bf A}$ is invertible
and the entries of ${\bf A}^{-1}$ are given in \eqref{mat-A-inverse}.
So we get
\begin{equation}\label{reg-exp-mat-inverse}
{\bf V}={\bf A}^{-1}(s_1-1,\ldots,s_r-1){\bf V}^\reg.
\end{equation}
Comparing the first entries of these two column matrices, we get
$$
(-1)^r \z(s_1,\ldots,s_r)
= \sum_{i \in I}(-1)^{r-i} a'_{0,i}(s_1-1, \ldots,s_r-1)
\z^\reg_{(a_{i+1}, \ldots, a_r)} (s_{i+1},\ldots,s_r).
$$
Now $a'_{0,i}(X_1, \ldots,X_r)=(-1)^{|I_i|}b_{0,i}(X_1, \ldots,X_r)
=(-1)^{|I_i|}f_{i}(X_1, \ldots,X_r)$,
hence the theorem follows.
\end{proof}

So \thmref{thm-reg-exp-inverse} yields a Laurent type expansion of $\z(s_1,\ldots,s_r)$
around any integer point $(a_1,\ldots,a_r)$ in $\overline{U_r}$ and
\eqref{multi-zeta-conv}, \eqref{multi-zeta-exp} and \eqref{gen-multi-zeta-exp}
can be considered as special cases of \thmref{thm-reg-exp-inverse}.
Generally the rational functions $f_{i}$ do not have simple expressions
as in \thmref{thm-gen-multi-zeta-exp}. This is shown by the following
example.

\begin{exm}\rm
When $r=3$ and $(a_1,a_2,a_3)=(2,0,1)$, we have $I=\{0,2,3\}$ and
\begin{equation*}
\begin{split}
\z^\reg_{(2,0,1)} (s_1,s_2,s_3) = & \ \z(s_1,s_2,s_3) +
\frac{1}{(s_2-1) (s_1+s_2-2)} \z(s_3)\\
&- \frac{1}{(s_3-1) (s_2+s_3-2)(s_1+s_2+s_3-3)}.
\end{split}
\end{equation*}
We therefore have,
\begin{equation*}
\begin{split}
\z(s_1,s_2,s_3) = & \ \z^\reg_{(2,0,1)} (s_1,s_2,s_3)-
\frac{1}{(s_2-1) (s_1+s_2-2)}\left(\z^\reg_{(1)} (s_3)+\frac{1}{s_3-1}\right)\\
& + \frac{1}{(s_3-1)(s_2+s_3-2)(s_1+s_2+s_3-3)}\\
= & \ \z^\reg_{(2,0,1)} (s_1,s_2,s_3)-
\frac{1}{(s_2-1)(s_1+s_2-2)}\z^\reg_{(1)} (s_3)\\
& - \frac{s_1+2s_2+s_3-4}{(s_2-1)(s_1+s_2-2)(s_2+s_3-2)(s_1+s_2+s_3-3)}.
\end{split}
\end{equation*}
Note that for $i \in I$,
$$
i-|I_i|= \begin{cases}
         0 & \text{ when } i=0,\\
         1 & \text{ when } i=2,3,
         \end{cases}
$$
and that for any positive real numbers $x_1,x_2,x_3$,
\begin{align*}
& \int\limits_{(t_1,t_2,t_3) \in [0,1]^3 \atop t_1>t_2, t_2<t_3}
t_1^{x_1-1} t_2^{x_2-1} t_3^{x_3-1} dt_1 dt_2 dt_3\\
& = \int\limits_{1>t_1>t_3>t_2>0} t_1^{x_1-1} t_2^{x_2-1} t_3^{x_3-1} dt_1 dt_2 dt_3
+\int\limits_{1>t_3>t_1>t_2>0} t_1^{x_1-1} t_2^{x_2-1} t_3^{x_3-1} dt_1 dt_2 dt_3\\
& = \frac{1}{x_2(x_2+x_3)(x_1+x_2+x_3)} + \frac{1}{x_2(x_1+x_2)(x_1+x_2+x_3)}\\
& = \frac{x_1+2x_2+x_3}{x_2(x_1+x_2)(x_2+x_3)(x_1+x_2+x_3)}.
\end{align*}
\end{exm}

\subsection{Stuffle product formula for regularised multiple zeta functions}

We begin with the notion of shuffling and stuffling of two non-negative integers $p,q$.
We define a \textit{stuffling} of $p$ and $q$ to be a triple $(r,A,B)$ of sets
such that $|A|=p, \ |B|=q$ and $A \cup B = \{1,\ldots,r\}$.
We then have $\max(p,q) \le r \le p+q$. We call $r$ to be the
{\it depth of the stuffling}. Such a stuffling is called a \textit{shuffling}
when $A$ and $B$ are disjoint, i.e. when $r = p+q$. Since in case of shuffling,
$r$ is uniquely determined by $|A|$ and $|B|$, we shall denote such
a shuffling just by $(A,B)$.

Let $(a_1,\ldots,a_p)$ and $(b_1,\ldots,b_q)$ be two sequences of
complex numbers and $(r,A,B)$ be a stuffling of $p$ and $q$.
Let $\sigma$ and $\tau$ denote the unique increasing bijections
from $A \to \{1,\ldots,p\}$ and $B \to\{1,\ldots,q\}$ respectively.
Let us define a sequence of complex numbers $(c_1,\ldots,c_r)$ as follows :
$$
c_i  := \begin{cases}
a_{\sigma(i)} & \text{when } i \in A \setminus B,\\
b_{\tau(i)} & \text{when } i \in B \setminus A, \\
a_{\sigma(i)} + b_{\tau(i)} & \text{when } i \in A \cap B.
\end{cases}
$$
We call $(c_1,\ldots,c_r)$, the \textit{sequence deduced from $(a_1,\ldots,a_p)$ and
$(b_1,\ldots,b_q)$ by the stuffling $(r,A,B)$}. Clearly,
if $(a_1,\ldots,a_p) \in \overline{U_p}$ and $(b_1,\ldots,b_q) \in \overline{U_q}$, then
$(c_1,\ldots,c_r) \in \overline{U_r}$.

It is well known that the multiple zeta functions satisfy the
{\it stuffle product formula}, i.e. product of two multiple zeta
functions can be expressed as sum of other multiple zeta functions,
where the sum runs over all possible stufflings of the coordinates
of the initial two multiple zeta functions. The simplest example
of this phenomenon can be given as the following equality of meromorphic
functions on $\C^2$  :
$$
\z(s_1) \z(s_2)=\z(s_1,s_2)+\z(s_2,s_1)+\z(s_1+s_2),
$$
which is due to Euler. Using \thmref{thm-reg-exp}, we now prove that the
regularised multiple zeta functions around integral points in the closure of the
domain of convergence also satisfy the stuffle product formula. More precisely,
we have the following theorem.

\begin{thm}\label{thm-reg-stuffle}
Let $p,q \ge 0$ be integers and $(a_1,\ldots,a_p) \in \overline{U_p} \cap \Z^p$,
$(b_1,\ldots,b_q) \in \overline{U_q} \cap \Z^q$. Then
we have the following equality of meromorphic functions on $\C^{p+q}:$
\begin{equation}\label{reg-stuffle}
\z^\reg_{(a_1, \ldots, a_p)} (s_1,\ldots,s_p) \cdot
\z^\reg_{(b_1, \ldots, b_q)} (t_1,\ldots,t_q)
=\sum_{(r,A,B)} \z^\reg_{(c_1, \ldots, c_r)} (u_1,\ldots,u_r),
\end{equation}
where the sum is over the stufflings $(r,A,B)$ of $p$ and $q$, and
$(u_1,\ldots,u_r),(c_1, \ldots, c_r)$ are the sequences deduced from
$(s_1,\ldots,s_p),(t_1,\ldots,t_q)$ and $(a_1, \ldots, a_p),(b_1, \ldots, b_q)$
respectively, by this stuffling.
\end{thm}

For the purpose of the proof, it is convenient to first prove two
combinatorial lemmas.

\begin{lem}\label{lem-rational-function-shuffle}
Let $p,q \ge 0$ be integers and $X_1,\ldots,X_p,Y_1,\ldots,Y_q$ be indeterminates.
Then we have the following equality of rational functions  $:$
\begin{equation}\label{rational-function-shuffle}
\begin{split}
&\frac{1}{X_p(X_p+X_{p-1})\cdots(X_p+\cdots+X_1)} \cdot
\frac{1}{Y_q(Y_q+Y_{q-1})\cdots(Y_q+\cdots+Y_1)}\\
&= \sum_{(A,B)} \frac{1}{Z_{p+q}(Z_{p+q}+Z_{p+q-1})\cdots(Z_{p+q}+\cdots+Z_1)},
\end{split}
\end{equation}
where the sum is over the shufflings $(A,B)$ of $p$ and $q$, and
$(Z_1,\ldots,Z_{p+q})$ is the sequence deduced from
$(X_1,\ldots,X_p),(Y_1,\ldots,Y_q)$ by this shuffling.
\end{lem}

\begin{proof}
It suffices to prove that this equality holds when we replace
the indeterminates $X_1,\ldots,X_p,Y_1,\ldots,Y_q$ by positive real
numbers $x_1,\ldots,x_p,y_1,\ldots,y_q$. The proof then follows from the
integral representation given in formula \eqref{a-ij}, as it is
well known \cite{RR} that the product of iterated integrals 
satisfy shuffle product formula.
\end{proof}

\begin{lem}\label{lem-bijection}
Let $p,q \ge 0$ be integers and $(a_1,\ldots,a_p) \in \overline{U_p} \cap \Z^p$,
$(b_1,\ldots,b_q) \in \overline{U_q} \cap \Z^q$. Let $E$ denote the set
of quadruples $(i,j,(I,J),(t,I',J'))$ satisfying the following conditions $:$
\begin{equation}\label{E}
\begin{split}
& 0 \le i \le p, 0 \le j \le q, a_1+\cdots+a_i=i, b_1+\cdots+b_j=j,\\
& (I,J) \ \text{is a shuffling of $i$ and $j$},\\
& (t,I',J') \ \text{is a stuffling of $(p-i)$ and $(q-j)$}.
\end{split}
\end{equation}
Let $F$ denote the set of pairs $((r,A,B),k)$ satisfying the following conditions $:$
\begin{equation}\label{F}
\begin{split}
& (r,A,B) \ \text{is a stuffling of $p$ and $q$},\\
& 0 \le k \le r \text{ and } c_1+\cdots+c_k=k, \text{ where }
(c_1,\ldots,c_r) \text{ is the sequence}\\
& \text{deduced from } (a_1,\ldots,a_p) \text{ and }
(b_1,\ldots,b_q) \text{ by this stuffling}.
\end{split}
\end{equation}
Then for each $(i,j,(I,J),(t,I',J')) \in E$,
$$
\varphi(i,j,(I,J),(t,I',J')) :=((i+j+t, I \cup (i+j+I'),J \cup (i+j+J')),i+j)
$$
belongs to $F$, and $\varphi$ is a bijection from $E$ to $F$.
\end{lem}

\begin{proof}
\underline{$\varphi(E) \subset F$} :
Note that $(i+j+I') \cup (i+j+J')=\{i+j+1,\ldots, i+j+t\}$. Now
since $(I,J)$ is a shuffling of $i$ and $j$ and $(t,I',J')$ is a stuffling
of $(p-i)$ and $(q-j)$, we get that $(i+j+t, I \cup (i+j+I'),J \cup (i+j+J'))$
is a stuffling of $p$ and $q$.

Now suppose $(c_1,\ldots,c_{i+j+t})$
is the sequence deduced from $(a_1,\ldots,a_p)$ and $(b_1,\ldots,b_q)$
by this stuffling. Since $(I,J)$ is a shuffling of $i$ and $j$ with
$a_1+\cdots+a_i=i$ and $b_1+\cdots+b_j=j$, we get that $c_1+\cdots+c_{i+j}=i+j$.
Hence, $\varphi(i,j,(I,J),(t,I',J')) \in F$.

\underline{$\varphi$ is injective} :
Let $\varphi(i_1,j_1,(I_1,J_1),(t_1,I_1',J_1'))
=\varphi(i_2,j_2,(I_2,J_2),(t_2,I_2',J_2'))$. Then
$i_1+j_1=i_2+j_2$ and $i_1+j_1+t_1=i_2+j_2+t_2$. Hence $t_1=t_2$.

Further, $I_1 \cup (i_1+j_1+I_1')=I_2 \cup (i_2+j_2+I_2')$. Since
$I_1, I_2$ are subsets of  $\{1,\ldots,i_1+j_1\}$, and
$(i_1+j_1+I_1'), (i_2+j_2+I_2')$ are subsets of $\{i_1+j_1+1,\ldots,i_1+j_1+t_1\}$,
we get that $I_1=I_2$ and $(i_1+j_1+I_1')=(i_2+j_2+I_2')$. Hence $i_1=i_2$ and
so $j_1=j_2$ and $I_1'=I_2'$. Similarly we get, $J_1=J_2$ and $J_1'=J_2'$.
This shows that $\varphi$ is injective.

\underline{$\varphi$ is surjective} :
Next let $((r,A,B),k) \in F$ and $(c_1,\ldots,c_r)$ be the sequence deduced from
$(a_1,\ldots,a_p)$ and $(b_1,\ldots,b_q)$ by the stuffling $(r,A,B)$.
Set $C=A \cap \{1,\ldots,k\}, D=B \cap \{1,\ldots,k\}$. We first prove that $(C,D)$
is a shuffling of $|C|$ and $|D|$. Clearly, $C \cup D =\{1,\ldots,k\}$.
We show that $C \cap D=\varnothing$. Let $\sigma$ and $\tau$ denote the
unique increasing bijections from $A \to \{1,\ldots,p\}$ and
$B \to\{1,\ldots,q\}$ respectively. Then $\sigma_{|C}  : C \to \{1,\ldots,|C|\}$
and $\tau_{|D}  : D \to \{1,\ldots,|D|\}$ denote the unique increasing bijections.
We then have, for $1 \le i \le k$,
$$
c_i  := \begin{cases}
a_{\sigma_{|C}(i)} & \text{when } i \in C \setminus D,\\
b_{\tau_{|D}(i)} & \text{when } i \in D \setminus C, \\
a_{\sigma_{|C}(i)} + b_{\tau_{|D}(i)} & \text{when } i \in C \cap D.
\end{cases}
$$
Hence, $c_1+\cdots+c_k=a_1+\cdots+a_{|C|}+b_1+\cdots+b_{|D|}$.
As $(a_1,\ldots,a_p) \in \overline{U_p} \cap \Z^p$ and
$(b_1,\ldots,b_q) \in \overline{U_q} \cap \Z^q$, we have
$a_1+\cdots+a_{|C|} \ge |C|$ and $b_1+\cdots+b_{|D|} \ge |D|$.
Hence $c_1+\cdots+c_k=k \ge |C|+|D|$. As $C \cup D =\{1,\ldots,k\}$,
we get $k = |C|+|D|$ and $C \cap D=\varnothing$.

Now set $C''=A \setminus \{1,\ldots,k\}$ and
$D''=B \setminus \{1,\ldots,k\}$. Putting
$C'=\{x-k  : x \in C''\}$ and
$D'=\{x-k  : x \in D''\}$, we get that
$$
\varphi(|C|,|D|,(C,D),(r-k,C',D'))=((r,A,B),k).
$$
This completes the proof of \lemref{lem-bijection}.
\end{proof}

\begin{proof}[Proof of \thmref{thm-reg-stuffle}]
We expand both the sides of \eqref{reg-stuffle}. Firstly,
from \thmref{thm-reg-exp} we get that
\begin{align*}
&\z^\reg_{(a_1, \ldots, a_p)} (s_1,\ldots,s_p) \cdot
\z^\reg_{(b_1, \ldots, b_q)} (t_1,\ldots,t_q)\\
&=\sum_{0 \le i \le p;\ a_1+\cdots+a_i=i \atop 0 \le j \le q;\ b_1+\cdots+b_j=j}
\frac{(-1)^{i+j} \z(s_{i+1},\ldots,s_p) \z(t_{j+1},\ldots,t_q)}
{(s_i-1) \cdots (s_i+\cdots+s_1-i) (t_j-1) \cdots (t_j+\cdots+t_1-j)}.
\end{align*}
Then using the stuffle product formula for multiple zeta functions and
\lemref{lem-rational-function-shuffle} we get that
\begin{align*}
&\z^\reg_{(a_1, \ldots, a_p)} (s_1,\ldots,s_p) \cdot
\z^\reg_{(b_1, \ldots, b_q)} (t_1,\ldots,t_q)\\
&=\sum_{(i,j,(I,J),(t,I',J')) \in E}
\frac{(-1)^{i+j} \z(u_{i+j+1},\ldots,u_{i+j+t})}
{(u_{i+j}-1) \cdots (u_{i+j}+\cdots+u_1-i-j)},
\end{align*}
where $E$ is as in \lemref{lem-bijection} and
$(u_1,\ldots,u_{i+j},u_{i+j+1},\ldots,u_{i+j+t})$ is the sequence
deduced from $(s_1,\ldots,s_p)$ and $(t_1,\ldots,t_q)$ by the stuffling
$(i+j+t,I\cup (i+j+I'),J \cup (i+j+J'))$.

Secondly, by \thmref{thm-reg-exp},
\begin{align*}
\sum_{(r,A,B)} \z^\reg_{(c_1, \ldots, c_r)} (z_1,\ldots,z_r)
&=\sum_{(r,A,B)} \sum_{0 \le k \le r \atop c_1+\cdots+c_k=k}
\frac{(-1)^{k} \z(z_{k+1},\ldots,z_r)}{(z_k-1) \cdots (z_k+\cdots+z_1-k)}\\
&=\sum_{((r,A,B),k)\in F}
\frac{(-1)^{k} \z(z_{k+1},\ldots,z_r)}{(z_k-1) \cdots (z_k+\cdots+z_1-k)},
\end{align*}
where $F$ is as in \lemref{lem-bijection}.
Now we conclude the theorem by \lemref{lem-bijection}.
\end{proof}

\thmref{thm-reg-stuffle} shows that the $\Q$-subspace of $\R$, generated by the
multiple Stieltjes constants $\gamma_{k_1, \ldots, k_r}^{(a_1,\ldots,a_r)}$, for
integers $r,k_1,\ldots,k_r\ge 0$, $(a_1,\ldots,a_r) \in \overline{U_r} \cap \Z^r$,
is a unitary $\Q$-subalgebra of $\R$. We denote it by $\Gamma$.
Its $\Q$-subspace generated by the multiple Stieltjes constants
$\gamma_{k_1, \ldots, k_r}^{(a_1,\ldots,a_r)}$, for integers
$r,k_1,\ldots,k_r\ge 0, a_1,\ldots,a_r\ge 1$ is a $\Q$-subalgebra
of $\Gamma$. We denote it by $\Gamma^+$.

\begin{cor}\label{cor-algebra}
Let $\Gamma^{++}$ be the $\Q$-subalgebra of $\Gamma^+$ generated by the
elements of the form $\gamma_{k_1, \ldots, k_r}^{(a_1,\ldots,a_r)}$,
where $r,k_1,\ldots,k_r\ge 0$ are integers and either all $a_i$'s are
equal to $1$, or $r,a_1,\ldots,a_r$ are positive integers with $a_1 \ge 2$.
Then $\Gamma^{++}=\Gamma^+$.
\end{cor}

\begin{proof}
Let $a_1,\ldots,a_r$ be any positive integers. If all $a_i$'s are not equal
to $1$, let $l=l(a_1,\ldots,a_r)$ be the integer such that $(a_1,\ldots,a_r)
=(1,\ldots,1,a_{l+1},\ldots,a_r)$ with $a_{l+1} \ge 2$. We prove
this corollary by induction on $l$. When $l=0$ or $l=r$, this is clear. When $0 <l < r$,
we deduce from \thmref{thm-reg-stuffle} that the product
$$
\gamma_{k_1, \ldots, k_l}^{(1,\ldots,1)} \gamma_{k_{l+1}, \ldots, k_r}^{(a_{l+1},\ldots,a_r)}
$$
is a $\Q$-linear combination of $\gamma_{k_1, \ldots, k_r}^{(a_1,\ldots,a_r)}$
and of other multiple Stieltjes constants that belong to $\Gamma^{++}$ by the
induction hypothesis. This completes the proof of \corref{cor-algebra}.
\end{proof}

\begin{rmk}\rm
Let $\Gamma'$ be the $\Q$-subalgebra of $\Gamma$, generated by the multiple
Stieltjes constants $\gamma_{k_1, \ldots, k_r}^{(a_1,\ldots,a_r)}$, for integers
$r,k_1,\ldots,k_r\ge 0$ and  $(a_1,\ldots,a_r)$ as in Section \ref{pos-boundary}. 
If $\Gamma''$ denotes the $\Q$-subalgebra of $\Gamma'$ generated by the
elements of the form $\gamma_{k_1, \ldots, k_r}^{(a_1,\ldots,a_r)}$,
where $r,k_1,\ldots,k_r\ge 0$ are integers and either all $a_i$'s are
equal to $1$, or $r \ge 1$ with $(a_1,\ldots,a_r) \in U_r \cap \Z^r$.
Then similarly we can derive that $\Gamma'=\Gamma''$.
\end{rmk}

\begin{rmk}\rm
A statement formally equivalent to \corref{cor-algebra}, has also been
stated in \cite[Theorem 1.3]{MOW}. But it does not imply \thmref{thm-reg-stuffle}.
\end{rmk}

\section{Asymptotic expansions of sequences of germs of holomorphic and 
meromorphic functions}\label{asymp-exp}

This is a preparatory section where we set up the language required to give a succinct
proof of \thmref{thm-gen-reg-exp}, which extends \thmref{thm-reg-exp} for
any general points with integral coordinates.

\subsection{Asymptotic expansions of sequences of germs of holomorphic functions}\label{asymp-holo}

Let ${\bf a}$ be a point in $\C^r$. Let $\cO_{\bf a}$ denote the
$\C$-algebra of germs of holomorphic functions at ${\bf a}$. For all $f \in \cO_{\bf a}$
and ${\bf k} \in \N^r$, let $c_{\bf k}(f)$ denote the ${\bf k}$-th
Taylor coefficient of $f$ at ${\bf a}$. In other words, $f$ is the germ at ${\bf a}$
of the function defined in some neighbourhood of ${\bf a}$ by
${\bf s} \mapsto \sum_{{\bf k} \in \N^r} c_{\bf k}(f) ({\bf s}-{\bf a})^{\bf k}$.

Let ${\bf T}=(T_1,\ldots,T_r)$ be a sequence of $r$ indeterminates.
The map $f \mapsto \sum_{{\bf k} \in \N^r} c_{\bf k}(f) {\bf T}^{\bf k}$ is an isomorphism
of $\C$-algebras from $\cO_{\bf a}$ to the algebra of convergent power series with
coefficients in $\C$ in the indeterminates ${\bf T}$.

We recall from Section \ref{msc} that the set $\cE$ of sequences 
$$
\left((\log n)^l n^{-m} \right)_{n\ge 1},
$$
where $l \in \N$ and $m \in \Z$, is a comparison scale on the set of
natural numbers $\N$, filtered by the Frechet filter. Let $A\in \Z$ be an integer.
We say that a sequence $(f_n)_{n \in \N}$ of elements of $\cO_{\bf a}$ has an
{\it asymptotic expansion to precision $n^{-A}$ relative to the comparison scale $\cE$}
if the following conditions are satisfied :\\
a) for each ${\bf k} \in \N^r$, the sequence of complex numbers $(c_{\bf k}(f_n))_{n \in \N}$
has an asymptotic expansion to precision $n^{-A}$ relative to $\cE$ of the form
$$
c_{\bf k}(f_n)=\sum_{l \in \N, m \in \Z  \atop m \le A} u_{({\bf k},l,m)} (\log n)^l n^{-m} + o(n^{-A}),
$$
when $n$ goes to $\infty$, where the family $(u_{({\bf k},l,m)})_{l \in \N, m \in \Z  \atop m \le A}$ has a finite support;\\
b) there exists $m_0 \in \Z$ such that $u_{({\bf k},l,m)}=0$ for ${\bf k} \in \N^r, l \in \N, m< m_0$;\\
c) for each $l \in \N$ and $m \in \Z, m \le A$, the power series
\begin{equation}\label{asymp-exp-ps}
\sum_{{\bf k} \in \N^r} u_{({\bf k},l,m)} ({\bf s}-{\bf a})^{\bf k}
\end{equation}
converges in some neighbourhood of ${\bf a}$ in $\C^r$.

When these conditions are satisfied and $g_{(l,m)}$ denotes the germ at ${\bf a}$
of the function defined by the power series \eqref{asymp-exp-ps}, then
$$
\sum_{l \in \N, m \in \Z  \atop m \le A} g_{(l,m)} L^l X^m
$$
is a Laurent polynomial in the indeterminate $X$ with coefficients in the formal
power series ring $\cO_{\bf a}[[L]]$. We call it {\it the formal asymptotic
expansion to precision $n^{-A}$
relative to $\cE$} of the sequence of germs $(f_n)_{n \in \N}$.

\begin{rmk}\rm
As in Section \ref{msc}, we can extend these definitions to sequences of germs $(f_n)$,
defined only for $n$ large enough.
\end{rmk}

\begin{rmk}\rm
Unlike in \rmkref{rmk-order-1}, for
a given $m \in \Z, m \le A$, the set of integers $l \in \N$ such that
$g_{(l,m)} \neq 0$ can be infinite. However, condition a) in Section \ref{asymp-holo}
implies that the order at ${\bf a}$ of $g_{(l,m)}$ goes to $\infty$ as $l$ tends to $\infty$.
\end{rmk}

\begin{exm}\label{exm-small-fn}\rm
Let $(f_n)_{n \in \N}$ be a sequence of elements of $\cO_{\bf a}$ satisfying
the following property : there exists an open neighbourhood $D$ of ${\bf a}$ and
a sequence $(F_n)_{n \in \N}$ of holomorphic functions on $D$, such that $f_n$
is the germ of $F_n$ at ${\bf a}$ and $\|F_n\|_D=o(n^{-A})$ for some integer $A$,
as $n$ tends to $\infty$. Then $(f_n)_{n \in \N}$ has an asymptotic expansion to
precision $n^{-A}$ relative to $\cE$ and its formal asymptotic expansion to
precision $n^{-A}$ is the Laurent polynomial $0$.

Indeed, for each ${\bf k} \in \N^r$, there exists a constant $\alpha_{\bf k}$
such that, for any bounded holomorphic function $F$ on $D$, the ${\bf k}$-th 
Taylor coefficient of $F$ at ${\bf a}$ is bounded by $\alpha_{\bf k}\|F\|_D$.
Hence the hypothesis implies that $c_{\bf k}(f_n)=o(n^{-A})$ as $n$ tends to $\infty$. 
\end{exm}

We say that a sequence $(f_n)_{n \in \N}$ of elements of $\cO_{\bf a}$ has {\it a
complete asymptotic expansion relative to $\cE$} if it has an asymptotic expansion to
precision $n^{-A}$ for all $A\in \Z$. In this case there exists a unique Laurent series
$$
G=\sum_{(l,m)\in \N \times \Z} g_{(l,m)} L^l X^m,
$$
in the indeterminate $X$ with coefficients in the formal power series ring
$\cO_{\bf a}[[L]]$ such that the Laurent polynomial obtained by truncating
$G$ to degree $\le A$ in $X$ is the formal asymptotic expansion of
$(f_n)_{n \in \N}$ to precision $n^{-A}$. We call $G$ {\it the formal 
complete asymptotic expansion} of the sequence of germs $(f_n)_{n \in \N}$
(relative to $\cE$). In this case we also get that for all  ${\bf k} \in \N^r$,
the formal asymptotic expansion of the sequence of complex numbers
$(c_{\bf k}(f_n))_{n \in \N}$ (relative to $\cE$) is given by
$$
\sum_{(l,m)\in \N \times \Z} c_{\bf k}(g_{(l,m)}) L^l X^m.
$$

The set $\cF$ of sequences of elements of $\cO_{\bf a}$ having a complete
asymptotic expansion relative to the comparison scale $\cE$ is a unitary
$\cO_{\bf a}$-subalgebra of $\cO_{\bf a}^\N$. The map that associates to
such a sequence its formal complete asymptotic expansion is a unitary
homomorphism of $\cO_{\bf a}$-algebras from $\cF$ to $\cO_{\bf a}[[L]]((X))$.

\begin{exm}\label{exm-n-power}\rm
Let us take $r=1$ and $a$ be an integer. For each $n \ge 2$, let $f_n$ denote the
germ at $a$ of the holomorphic function $s \mapsto n^{1-s}$. The sequence
$(f_n)_{n \ge 2}$ has a complete asymptotic expansion, and its formal complete
asymptotic expansion is 
$$
G=\sum_{(l,m)\in \N \times \Z} g_{(l,m)} L^l X^m,
$$
where $g_{(l,m)}=0$ if $m \neq 1-a$ and $g_{(l,1-a)}$ is the germ at $a$ of the
holomorphic function $s \mapsto \frac{(-1)^l}{l!}(s-a)^l$.

Indeed, for each $k \ge 0$ and $n \ge 2$, we have $c_k(f_n)=\frac{(-1)^k (\log n)^k}{k!} n^{1-a}$,
i.e. condition a) holds, with $u_{(k,l,m)}=0$ if $(l,m) \neq (k,1-a)$ and $u_{(k,k,1-a)}=\frac{(-1)^k}{k!}$.
\end{exm}

\begin{lem}\label{lem-prod-exp}
Let $(f_n)_{n \in \N}$ be a sequence of elements of $\cO_{\bf a}$
and $f$ be a non-zero element in $\cO_{\bf a}$. If the sequence
$(f f_n)_{n \in \N}$ has  an asymptotic expansion to precision $n^{-A}$,
for a given integer $A$ (resp. a complete asymptotic expansion)
relative to the comparison scale $\cE$, then the same
holds for the sequence $(f_n)_{n \in \N}$.
\end{lem}

\begin{proof}
It is sufficient to prove the first statement.
For this we first show that for each ${\bf k}=(k_1,\ldots,k_r) \in \N^r$, the sequence
of complex numbers $(c_{\bf k}(f_n))_{n \in \N}$ has an
asymptotic expansion to precision $n^{-A}$. We do this by induction
on $d=|{\bf k}| :=k_1+\cdots+k_r$.
We know by the induction hypothesis that the result holds 
for the components of $f_n$ of order $<d$. By removing these components
from $f_n$, we can assume that all $f_n$ are of order $\ge d$.
Now denote by $d'$ the order of $f$ and by $P$ the
homogeneous component of degree $d'$ in $f$. The map $Q \mapsto PQ$ is an
injective linear map from the vector space of homogeneous polynomials
with complex coefficients of degree $d$ to the vector space of homogeneous polynomials
with complex coefficients of degree $d+d'$. Hence it has a linear
retraction and therefore there exists a family of complex
numbers $(t_{\bf j})_{{\bf j} \in \N^r \atop |{\bf j}|=d+d'}$
such that for every germ $g \in \cO_{\bf a}$ of order $\ge d$, we have
$$
c_{\bf k}(g)= \sum_{{\bf j} \in \N^r \atop |{\bf j}|=d+d'} t_{\bf j} c_{\bf j}(fg).
$$
This, in particular, applies to the germs $f_n$. Since the sequences
$(c_{\bf j}(f f_n))_{n \in \N}$ have asymptotic expansions to precision $n^{-A}$,
the sequence $(c_{\bf k}(f_n))_{n \in \N}$ also has one.

From the above argument we also get that if the formal asymptotic expansion 
to precision $n^{-A}$ of the sequence of germs $(f f_n)_{n \in \N}$ does not
have terms of degree $< m_0$ in $X$, the same is true for the formal asymptotic
expansion to precision $n^{-A}$ of the sequence of complex numbers
$(c_{\bf k}(f_n))_{n \in \N}$.
Now if we denote by $\hat \cO_{\bf a}$ the completion of $\cO_{\bf a}$
relative to the $\mathfrak m$-adic topology, where $\mathfrak m$ is
its unique maximal ideal,
$\hat \cO_{\bf a}$ is isomorphic to the algebra of formal power series with
coefficients in $\C$ and indeterminates ${\bf T}$.
Hence we get a Laurent series
$$
G=\sum_{(l,m)\in \N \times \Z \atop m \le A} g_{(l,m)} L^l X^m,
$$
in $\hat \cO_{\bf a}[[L]]((X))$ such that for each ${\bf k} \in \N^r$,
the formal asymptotic expansion to precision $n^{-A}$ of the sequence of
complex numbers $(c_{\bf k}(f_n))_{n \in \N}$ (relative to $\cE$) is given by
$$
\sum_{(l,m)\in \N \times \Z \atop m \le A} c_{\bf k}(g_{(l,m)}) L^l X^m.
$$
But then the formal asymptotic expansion to precision $n^{-A}$ of the sequence
of germs $(f f_n)_{n \in \N}$ is 
$$
\sum_{(l,m)\in \N \times \Z \atop m \le A} fg_{(l,m)} L^l X^m,
$$
which implies that $fg_{(l,m)} \in \cO_{\bf a}$ for all $l\in \N$ and $m \in \Z, m \le A$.
Now since $\hat \cO_{\bf a}$ is faithfully flat on $\cO_{\bf a}$
(see \cite[Corollary 7.3.5, p. 67]{AG}), we deduce that
$g_{(l,m)} \in \cO_{\bf a}$ for all $l\in \N$ and $m \in \Z$, by the following general lemma.
This completes the proof.
\end{proof}

\begin{lem}
Let $\phi  : A \to B$ be a faithfully flat commutative ring extension.
Let $a\in A$ be such that $a$ is not a zero divisor in $A$, and $b \in B$.
If $ab \in \phi(A)$, then $b \in \phi(A)$.
\end{lem}

\begin{proof}
Since $B$ is  faithfully flat over $A$, $\phi$ is injective. We may identify
$A$ with a subring of $B$, and $B/A$ is then a flat $A$-module
(see \cite[Chap. 1, \S3.5, Prop. 9]{NB-CA}). Since
$a$ is not a zero divisor, multiplication by $a$ is an injective $A$-linear map
from $A$ to $A$ and therefore induces an injective $A$-linear map from
$B/A$ to $B/A$. Since $ab \in A$, the class of $b$ in $B/A$ is in the 
kernel of this map, which therefore implies that $b \in A$ as the map
is injective.
\end{proof}

\subsection{Asymptotic expansions of sequences of germs of meromorphic functions}\label{asymp-mero}
Again let ${\bf a}$ be a point in $\C^r$. Let $\cM_{\bf a}$ denote the
$\C$-algebra of germs of meromorphic functions at ${\bf a}$. It is identified
to the field of fractions of the integral domain $\cO_{\bf a}$.

For an integer $A \in \Z$, we say that a sequence $(f_n)_{n \in \N}$ of elements of
$\cM_{\bf a}$ has an {\it asymptotic expansion to precision $n^{-A}$}
(resp. a {\it complete asymptotic expansion}) relative to $\cE$
if there exists a common denominator $f$ of $f_n$ (i.e. a non-zero element $f$
of $\cO_{\bf a}$ such that $fh_n \in \cO_{\bf a}$ for all $n \in \N$), such that
the sequence $(fh_n)_{n \in \N}$ of elements of $\cO_{\bf a}$ has an 
asymptotic expansion to precision $n^{-A}$ (resp. a complete
asymptotic expansion) relative to $\cE$, in the sense of Section \ref{asymp-holo}.

If this condition is satisfied for a common denominator $f$ of $f_n$'s,
then, by \lemref{lem-prod-exp}, it also holds for any other common
denominators of the $f_n$'s.
It follows that $(f_n)_{n \in \N}$ has a complete asymptotic expansion if
and only if it has an asymptotic expansion to precision $n^{-A}$ for all $A\in \Z$.
Moreover, the formal Laurent series
$$
\sum_{(l,m)\in \N \times \Z} f^{-1}g_{(l,m)} L^l X^m,
$$
in the indeterminate $X$ with coefficients in $\cM_{\bf a}[[L]]$, where
$$
\sum_{(l,m)\in \N \times \Z} g_{(l,m)} L^l X^m
$$
is the formal complete asymptotic expansion of the sequence $(ff_n)_{n \in \N}$,
does not depend on the choice of the common denominator $f$. This element
of the ring $\cM_{\bf a}[[L]]((X))$ is then called the {\it formal complete asymptotic
expansion} of the sequence of germs of meromorphic functions $(f_n)_{n \in \N}$
relative to $\cE$.

\begin{rmk}\label{rmk-defn}\rm
When all the $f_n$'s belong to $\cO_{\bf a}$, we can take $f=1$, and therefore
the definitions given in this section agree with those of Section \ref{asymp-holo}.
\end{rmk}

The set $\cH$ of sequences of elements of $\cM_{\bf a}$ having a complete
asymptotic expansion relative to the comparison scale $\cE$ is a unitary
$\cM_{\bf a}$-subalgebra of $\cM_{\bf a}^\N$. The map that associates to such
a sequence its formal complete asymptotic expansion is a unitary
$\cM_{\bf a}$-algebra homomorphism from $\cH$ to $\cM_{\bf a}[[L]]((X))$.

\begin{rmk}\label{rmk-variable} \rm
Let ${\bf a} \in \C^r$ and $\pi$ be the germ at ${\bf a}$ of a holomorphic map
defined on a neighbourhood of ${\bf a}$ with values in  $\C^p$ for some $p \ge 0$.
Denote ${\bf b}=\pi({\bf a})$. Suppose that $\pi$ is {\it dominant} at ${\bf a}$, i.e.
the ring homomorphism $f \mapsto f \circ \pi$ from $\cO_{\C^p,\bf b}$ to $\cO_{\C^r,\bf a}$
is injective. This homomorphism then extends to a field homomorphism
from $\cM_{\C^p,\bf b}$ to $\cM_{\C^r,\bf a}$, that we still denote by $f \mapsto f \circ \pi$.

Let $(f_n)_{n \in \N}$ be a sequence of elements of $\cM_{\C^p,\bf b}$ which has an
asymptotic expansion to precision $n^{-A}$, with $A\in \Z$. Then the sequence
$(f_n \circ \pi)_{n \in \N}$ of elements of $\cM_{\C^r,\bf a}$ also has an asymptotic
expansion to precision $n^{-A}$. Moreover, if
$\sum_{(l,m)\in \N \times \Z \atop m \le A} g_{(l,m)} L^l X^m$ is the formal
asymptotic expansion of $(f_n)_{n \in \N}$ to precision $n^{-A}$, then
$\sum_{(l,m)\in \N \times \Z \atop m \le A} (g_{(l,m)}\circ \pi) L^l X^m$
is that of $(f_n \circ \pi)_{n \in \N}$.
\end{rmk}

\subsection{Asymptotic expansions of $\z(s)_{> N}$ and $\z^\star(s)_{\ge N}$}\label{asymp-zeta-tail}
Recall that $s\mapsto \z(s)_{> N} := \sum_{n > N}n^{-s}$ is
holomorphic in the half plane $\Re(s)>1$ and has a meromorphic
extension to $\C$. In this section, we shall prove the following proposition.

\begin{prop}\label{prop-zeta-tail-asymp}
Let $a \in \Z$ be an integer. The sequence of germs at $a$ of meromorphic functions
$(\z(s)_{> N})_{N \ge 2}$ has a complete asymptotic expansion relative to $\cE$,
in the sense of Section \ref{asymp-mero}. The associated formal complete asymptotic
expansion is the formal Laurent series
$$
\sum_{k \ge 0} \sum_{l \ge 0} h_{(l,k)} L^l X^{1-a-k},
$$
where $h_{(l,k)}$ is the germ at $a$ of the function $s \mapsto \frac{(-1)^{l}B_k}{l!k!}(s)_{k-1}(s-a)^l$,
and for $k \ge 0$, $(s)_k$ denotes the Pochhammer symbol $s(s+1)\ldots (s+k-1)$
with $(s)_{-1} :=\frac{1}{s-1}$.
\end{prop}

\begin{proof}
We choose a positive integer $A$, and first prove that the sequence
$(\z(s)_{> N})_{N \ge 2}$ has an asymptotic expansion to precision $N^{-A}$.
Let $k_0$ be the smallest non-negative integer such that $a+k_0>A+1$ and
$D$ be an open disc of radius $\rho<1$ around $a$. We argue by induction
on $k_0$. If $k_0=0$, then $a \ge A+2$ and therefore, $\|\z(s)_{> N}\|_D=o(N^{-A})$
as $N$ tends to $\infty$. Hence $(\z(s)_{> N})_{N \ge 2}$ has an asymptotic
expansion to precision $N^{-A}$, by \exmref{exm-small-fn}.

Now suppose $k_0 \ge 1$. Consider the following identity (see \cite[Eq. (9)]{MSV}),
which is valid for any integer $n \ge 2$ and complex number $s$  :
\begin{equation}\label{first-step-zeta}
(n-1)^{1-s} - n^{1-s}= \sum_{k \ge 0} \frac{(s-1)_{k+1}}{(k+1)!} \ n^{-s-k}.
\end{equation}
Summing this for $n > N$, we deduce that
\begin{equation}\label{zeta-translate-1}
N^{1-s} =\sum_{k\ge 0} \frac{(s-1)_{k+1}}{(k+1)!} \zeta(s+k)_{> N},
\end{equation}
as an equality between holomorphic functions in $\Re(s)>1$.
The interchange of summation on the right hand side is justified
by \lemref{lem-negligible}, more simply by \cite[Prop. 2]{MSV}.
The equality in \eqref{zeta-translate-1} extends as an identity of meromorphic
functions on the whole of $\C$. Note that $\zeta(s+k)_{> N}$ is holomorphic on
$D$ for any $k > 1-a$. Then from \lemref{lem-negligible}, we get that the sum
$$
\sum_{k \ge k_0} \left\| \frac{(s-1)_{k+1}}{(k+1)!}
\z(s+k)_{> N}\right\|_D
$$
exists and it is $o(N^{-A})$ as $N$ tends to $\infty$.
Using Examples \ref{exm-small-fn} and \ref{exm-n-power}, we therefore get that
the sequence of germs at $a$ of meromorphic functions
$$
\left(\sum_{k=0}^{k_0-1} \frac{(s-1)_{k+1}}{(k+1)!} \zeta(s+k)_{> N}\right)_{N \ge 2}
$$
has an asymptotic expansion to precision $N^{-A}$, with the same formal
asymptotic expansion to precision $N^{-A}$ as that of the sequence $(N^{1-s})_{N \ge 2}$.

For $1 \le k \le k_0-1$, the sequence of germs $(\z(s)_{> N})_{N \ge 2}$ 
at $a+k$ has an asymptotic expansion to precision $N^{-A}$, by the induction hypothesis.
Thus for $1 \le k \le k_0-1$, the sequence of germs at $a$ of meromorphic functions
$(\z(s+k)_{> N})_{N \ge 2}$ has an asymptotic expansion to precision $N^{-A}$,
and hence the same holds for the sequence of germs $(\z(s)_{> N})_{N \ge 2}$
at $a$ from the above observation.

We shall now write down this expansion. Following above arguments we get that
for $0 \le j < k_0$, the sequences of germs at $a$ of meromorphic functions
$(N^{1-s-j})_{N \ge 2}$ and
$(\sum_{k=0}^{k_0-1-j} \frac{(s+j-1)_{k+1}}{(k+1)!} \zeta(s+j+k)_{> N})_{N \ge 2}$
have the same formal asymptotic expansion to precision $N^{-A}$.
This can be summarised by the following matrix identity  :
\begin{equation}\label{zeta-mat-1}
{\bf W} = {\bf A} {\bf V},
\end{equation}
where ${\bf V}, {\bf W}$ are column vectors whose entries are the formal asymptotic
expansions at $a$ to precision $N^{-A}$ of the column vectors
$$
\left[\begin{array}{c}
\z(s)_{> N}\\ 
\z(s+1)_{> N}\\
\vdots\\
\z(s+k_0-1)_{> N}
\end{array}\right]
\ \text{and} \
\left[\begin{array}{c}
N^{1-s}\\ 
N^{-s}\\
\vdots\\
N^{2-s-k_0}
\end{array}\right]
$$
respectively, and ${\bf A}$ is the square matrix, whose entries are the germs at $a$
of the following rational functions of $s$ :
\begin{equation}\label{mat-A}
\left[\begin{array}{c c c c}
(s-1)_1 & \frac{(s-1)_2}{2!} & \cdots & \frac{(s-1)_{k_0}}{k_0!}\\ 
0 & (s)_1 & \cdots & \frac{(s)_{k_0-1}}{(k_0-1)!}\\ 
\vdots & \vdots & \ddots & \vdots\\
0 & 0 & \cdots & (s+k_0-2)_1
\end{array}\right].
\end{equation}
The matrix ${\bf A}$ is invertible (see \cite[p. 494-495]{MSV}) and its inverse matrix
${\bf B}$ consists of the germs at $a$ of the following rational functions of $s$ :
\begin{equation}\label{mat-B}
\left[\begin{array}{c c c c c}
\frac{1}{s-1} & B_1 & \frac{(s)_1 B_2}{2!} & \cdots
 & \frac{(s)_{k_0-2} B_{k_0-1}}{(k_0-1)!}\\
0 & \frac{1}{s} & B_1 & \cdots & \frac{(s+1)_{k_0-3} B_{k_0-2}}{(k_0-2)!}\\ 
0 & 0 & \frac{1}{s+1} & \cdots & \frac{(s+2)_{k_0-4} B_{k_0-3}}{(k_0-3)!}\\
\vdots & \vdots & \vdots & \ddots &\vdots\\
0 & 0 & 0 & \cdots & \frac{1}{s+k_0-2}
\end{array}\right],
\end{equation}
where for $n \ge 0$, $B_n$'s are the Bernoulli numbers, defined by the following
generating series :
\begin{equation}\label{Bernoulli}
\frac{x}{e^x-1}=\sum_{n \ge 0} B_n \frac{x^n}{n!}.
\end{equation}
So we can rewrite \eqref{zeta-mat-1} as
\begin{equation}\label{zeta-mat-2}
{\bf V} = {\bf B} {\bf W}.
\end{equation}
The first entry of this matrix identity yields that the formal asymptotic expansion
at $a$ to precision $N^{-A}$ of $(\z(s)_{> N})_{N \ge 2}$ is same as that of 
$$
\left(\sum_{k=0}^{k_0-1} \frac{B_k}{k!} (s)_{k-1}N^{1-s-k} \right)_{N \ge 2}.
$$
From \exmref{exm-n-power}, we know that for $k \ge 0$, the formal complete asymptotic expansion
of the sequence of the germs at $a$ of functions $(N^{1-s-k})_{N \ge 2}$ is
$$
\sum_{(l,m)\in \N \times \Z} g_{(l,m)} L^l X^m,
$$
where $g_{(l,m)}=0$ if $m \neq 1-a-k$ and $g_{(l,1-a-k)}$ is the germ at $a$ of the
holomorphic function $s \mapsto \frac{(-1)^l}{l!}(s-a)^l$.
Thus we get that the formal asymptotic expansion to precision
$N^{-A}$ of the sequence of germs at $a$ of meromorphic functions
$(\z(s)_{> N})_{N \ge 2}$ is
$$
\sum_{k=0}^{k_0-1} \sum_{l \ge 0} h_{(l,k)} L^l X^{1-a-k},
$$
where $h_{(l,k)}$ is the germ at $a$ of the
holomorphic function $s \mapsto \frac{(-1)^{l}B_k}{l!k!}(s)_{k-1}(s-a)^l$.
Since this is true for any positive integer $A$, we conclude the proof.
\end{proof}

\begin{rmk}[Asymptotic expansions of $\z^\star(s)_{\ge N}$]\label{rmk-asymp-zeta-star-tail}\rm
Our arguments above can be readily used to find out the formal
complete asymptotic expansion of the sequence of germs at $a$
of meromorphic functions $(\z^\star(s)_{\ge N})_{N \ge 2}$, where
for $\Re(s)>1$, $\z^\star(s)_{\ge N} := \sum_{n \ge N}n^{-s}$.
If $B_n^\star$ denotes the {\it star Bernoulli numbers}, defined by
$B_n^\star :=(-1)^n B_n$, the formal complete asymptotic
expansion of the sequence of germs at $a$ of $(\z^\star(s)_{\ge N})_{N \ge 2}$
relative to $\cE$ is the formal Laurent series
$$
\sum_{k \ge 0} \sum_{l \ge 0} h^\star_{(l,k)} L^l X^{1-a-k},
$$
where $h^\star_{(l,k)}$ is the germ at $a$ of the function
$s \mapsto \frac{(-1)^{l}B^\star_k}{l!k!}(s)_{k-1}(s-a)^l$.

Indeed, the relevant formula, analogous to \eqref{zeta-translate-1},
that we need in this case is \eqref{second-step-star-tail-r=1} i.e.
\begin{equation}\label{zeta-star-translate-1}
N^{1-s} =\sum_{k\ge 0}(-1)^k \frac{(s-1)_{k+1}}{(k+1)!} \zeta(s+k)_{\ge N}.
\end{equation}
Then in the matrix ${\bf B}$ (or, equivalently in \eqref{mat-B}), the
Bernoulli numbers $B_n$'s will be replaced by the star Bernoulli
numbers $B_n^\star$.
\end{rmk}

\subsection{Asymptotic expansions of $\z(s_1,\ldots,s_r)_{> N}$
and $\z^\star(s_1,\ldots,s_r)_{\ge N}$}\label{asymp-multi-zeta-tail}
We recall \eqref{zeta-tail}, that for $(s_1,\ldots,s_r) \in U_r$
$$
\z(s_1,\ldots,s_r)_{> N} := \sum_{n_1 > \cdots > n_r > N}
n_1^{-s_1} \cdots n_r^{-s_r}.
$$
This function is holomorphic in $U_r$ and has a meromorphic extension to $\C^r$.
In this section we shall prove the following proposition, which generalises
\propref{prop-zeta-tail-asymp}.

\begin{prop}\label{prop-multi-zeta-tail-asymp}
Let $r \ge 1$ be an integer and ${\bf a}=(a_1,\ldots,a_r)\in \Z^r$. The sequence
of germs at ${\bf a}$ of meromorphic functions $(\z(s_1,\ldots,s_r)_{> N})_{N \ge 2}$
has a complete asymptotic expansion relative to $\cE$, in the sense of Section \ref{asymp-mero}.
The associated formal complete asymptotic
expansion is the formal Laurent series
$$
\sum_{{\bf k} \in \N^r} \sum_{l \ge 0} h_{(l,{\bf k})} L^l X^{r-|{\bf a}|-|{\bf k}|}
$$
where $h_{(l,{\bf k})}$ is the germ at ${\bf a}$ of the function
\begin{align*}
{\bf s}=(s_1,\ldots,s_r) \mapsto & \frac{(-1)^{l}B_{k_1} \cdots B_{k_r}}{l! k_1!\cdots k_r!}
(s_1)_{k_1-1} (s_1+s_2+k_1-1)_{k_2-1} \cdots \\
& \times (s_1+\cdots+s_r+k_1+\cdots+k_{r-1}-r+1)_{k_r-1} (|{\bf s}|-|{\bf a}|)^l.
\end{align*}
\end{prop}

\begin{proof}
We prove this by induction on $r$. The case $r=1$ has been treated in
\propref{prop-zeta-tail-asymp}. So we assume $r \ge 2$. 
Now we choose a positive integer $A$, and as in \propref{prop-zeta-tail-asymp},
first prove that the sequence $(\z(s_1,\ldots,s_r)_{> N})_{N \ge 2}$
has an asymptotic expansion to precision $N^{-A}$.

For an integer $k \in \Z$, let $\tau_k({\bf a})$ denote the point $(a_1+k,a_2,\ldots,a_r)\in \Z^r$.
Now let $k_0$ be the smallest non-negative integer such that $\tau_{k_0}({\bf a}) \in \tau_{A}(U_r)$
and $D$ be an open polydisc of polyradius $(\rho_1,\ldots,\rho_r)$ around ${\bf a}$
such that $\rho_1+\cdots+\rho_r<1$. Now if $k_0=0$, then ${\bf a} \in \tau_{A}(U_r)$
and therefore $\|\z(s_1,\ldots,s_r)_{> N}\|_D=o(N^{-A})$ 
as $N$ tends to $\infty$. Hence $(\z(s_1,\ldots,s_r)_{> N})_{N \ge 2}$ has an asymptotic
expansion to precision $N^{-A}$, by \exmref{exm-small-fn}. 

We now argue by induction on $k_0$ and suppose $k_0 \ge 1$.
Starting from \eqref{first-step-zeta}, we deduce that
\begin{equation}\label{multi-zeta-translate-1}
\z(s_1+s_2-1, s_3,\ldots,s_r)_{> N}
=\sum_{k\ge 0} \frac{(s_1-1)_{k+1}}{(k+1)!} \z(s_1+k,s_2,\ldots,s_r)_{> N},
\end{equation}
as an equality between holomorphic functions in $U_r$.
This then extends as an identity of meromorphic functions on $\C^r$.
Note that $\z(s_1+k,s_2,\ldots,s_r)_{> N}$ is holomorphic on $D$ for any
$k$ such that $\tau_k({\bf a}) \in U_r$.
Then from \lemref{lem-negligible}, we get that the sum
$$
\sum_{k \ge k_0} \left\| \frac{(s_1-1)_{k+1}}{(k+1)!} \z(s_1+k,s_2,\ldots,s_r)_{> N}\right\|_D
$$
exists and it is $o(N^{-A})$ as $N$ tends to $\infty$. By the induction hypothesis (for depth $r-1$)
and \rmkref{rmk-variable}, the sequence of germs at ${\bf a}$ of meromorphic functions
$(\z(s_1+s_2-1, s_3,\ldots,s_r)_{> N})_{N \ge 2}$ has an asymptotic expansion to
precision $N^{-A}$. This together with \exmref{exm-small-fn}, yields that the sequence of
germs at ${\bf a}$ of the meromorphic functions
$$
\left(\sum_{k=0}^{k_0-1} \frac{(s_1-1)_{k+1}}{(k+1)!} \z(s_1+k,s_2,\ldots,s_r)_{> N}\right)_{N \ge 2}
$$
has an asymptotic expansion to precision $N^{-A}$, and that
the associated formal asymptotic expansion to precision $N^{-A}$
is same as that of $(\z(s_1+s_2-1, s_3,\ldots,s_r)_{> N})_{N \ge 2}$.

For $1 \le k \le k_0-1$, the sequence of germs $(\z(s_1,\ldots,s_r)_{> N})_{N \ge 2}$ 
at $\tau_k({\bf a})$ has an asymptotic expansion to precision $N^{-A}$, by the induction hypothesis
(for $k < k_0$). Thus for $1 \le k \le k_0-1$, the sequence of germs at ${\bf a}$ of meromorphic functions
$(\z(s_1+k,\ldots,s_r)_{> N})_{N \ge 2}$ has an asymptotic expansion to precision $N^{-A}$,
and hence the same holds for the sequence of germs $(\z(s_1,\ldots,s_r)_{> N})_{N \ge 2}$
at ${\bf a}$ from the above observation.

We shall now write down this expansion. Arguments as above show that
for $0 \le j < k_0$, the sequences of germs at ${\bf a}$ of meromorphic functions
$(\z(s_1+s_2-1+j, s_3,\ldots,s_r)_{> N})_{N \ge 2}$ and
$(\sum_{k=0}^{k_0-1-j} \frac{(s_1+j-1)_{k+1}}{(k+1)!} \z(s_1+j+k,s_2,\ldots,s_r)_{> N})_{N \ge 2}$
have the same formal asymptotic expansion to precision $N^{-A}$.
This can be summarised by the following matrix identity :
\begin{equation}\label{multi-zeta-mat-1}
{\bf Y} = {\bf A} {\bf X},
\end{equation}
where ${\bf X}, {\bf Y}$ are column vectors whose entries are the formal asymptotic
expansions at ${\bf a}$ to precision $N^{-A}$ of the column vectors
$$
\left[\begin{array}{c}
\z(s_1,\ldots,s_r)_{> N}\\ 
\z(s_1+1,\ldots,s_r)_{> N}\\
\vdots\\
\z(s_1+k_0-1,\ldots,s_r)_{> N}
\end{array}\right] 
\ \text{and} \
\left[\begin{array}{c}
\z(s_1+s_2-1, s_3,\ldots,s_r)_{> N}\\ 
\z(s_1+s_2, s_3,\ldots,s_r)_{> N}\\
\vdots\\
\z(s_1+s_2+k_0-2, s_3,\ldots,s_r)_{> N}
\end{array}\right]
$$
respectively, and ${\bf A}$ is the square matrix, whose entries are the germs at ${\bf a}$
of the rational functions in \eqref{mat-A}, with $s$ replaced by $s_1$.
We have already seen that the matrix ${\bf A}$ is invertible and its inverse matrix
${\bf B}$ consists of the germs at ${\bf a}$ of rational functions in \eqref{mat-B},
with $s$ replaced by $s_1$. So we can rewrite \eqref{multi-zeta-mat-1} as
\begin{equation}\label{multi-zeta-mat-2}
{\bf X} = {\bf B} {\bf Y}.
\end{equation}
The first entry of this matrix identity yields that the formal asymptotic expansion
at ${\bf a}$ to precision $N^{-A}$ of $(\z(s_1,\ldots,s_r)_{> N})_{N \ge 2}$
is same as that of 
$$
\left(\sum_{k=0}^{k_0-1} \frac{B_k}{k!} (s_1)_{k-1}\z(s_1+s_2-1+k, s_3,\ldots,s_r)_{> N} \right)_{N \ge 2}.
$$
By the induction hypothesis (for depth $r-1$) and \rmkref{rmk-variable}, we get that for $k \ge 0$,
the formal complete asymptotic expansion of the sequence 
$(\z(s_1+s_2-1+k, s_3,\ldots,s_r)_{> N})_{N \ge 2}$ of the germs at ${\bf a}$ is
$$
\sum_{k_2,\ldots,k_r \ge 0} \sum_{l \ge 0} h^{(k)}_{(l,k_2,\ldots,k_r )} L^l X^{r-|{\bf a}|-k-k_2-\cdots-k_r},
$$
where $h^{(k)}_{(l,k_2,\ldots,k_r )}$ is the germ at ${\bf a}$ of the function
\begin{align*}
{\bf s} \mapsto & \frac{(-1)^{l}B_{k_2} \cdots B_{k_r}}{l! k_2!\cdots k_r!}
(s_1+s_2+k-1)_{k_2-1} \cdots \\
& \times (s_1+\cdots+s_r+k+k_2+\cdots+k_{r-1}-r+1)_{k_r-1} (|{\bf s}|-|{\bf a}|)^l.
\end{align*}
From this we therefore get the desired conclusion.
\end{proof}

\begin{rmk}[Asymptotic expansions of $\z^\star(s_1,\ldots,s_r)_{\ge N}$]
\label{rmk-asymp-multi-zeta-star-tail}\rm
Our arguments above can again be used to prove that the sequence of
germs at ${\bf a}$ of meromorphic functions $(\z^\star(s_1,\ldots,s_r)_{\ge N})_{N \ge 2}$,
where for $(s_1,\ldots,s_r) \in U_r$, $ \z^\star(s_1,\ldots,s_r)_{\ge N} :=
\sum_{n_1 \ge \cdots \ge n_r \ge N} n_1^{-s_1} \cdots n_r^{-s_r}$,
has a complete asymptotic expansion relative to $\cE$. The corresponding
formal complete asymptotic expansion is the formal Laurent series
$$
\sum_{{\bf k} \in \N^r} \sum_{l \ge 0} h^\star_{(l,{\bf k})} L^l X^{r-|{\bf a}|-|{\bf k}|},
$$
where $h^\star_{(l,{\bf k})}$ is the germ at ${\bf a}$ of the function
\begin{align*}
{\bf s}=(s_1,\ldots,s_r) \mapsto & \frac{(-1)^{l}B^\star_{k_1} \cdots B^\star_{k_r}}{l! k_1!\cdots k_r!}
(s_1)_{k_1-1} (s_1+s_2+k_1-1)_{k_2-1} \cdots  \\
& \times (s_1+\cdots+s_r+k_1+\cdots+k_{r-1}-r+1)_{k_r-1} (|{\bf s}|-|{\bf a}|)^l.
\end{align*}

Indeed, the relevant formula, analogous to \eqref{multi-zeta-translate-1},
that we need in this case is \eqref{second-step-star-tail} i.e.
\begin{equation}\label{multi-zeta-star-translate-1}
\z^\star(s_1+s_2-1, s_3,\ldots,s_r)_{\ge N}
=\sum_{k\ge 0}(-1)^k \frac{(s_1-1)_{k+1}}{(k+1)!} \z^\star(s_1+k,s_2,\ldots,s_r)_{\ge N},
\end{equation}
and then in the matrix ${\bf B}$ (or, equivalently in \eqref{mat-B}), the
Bernoulli numbers $B_n$'s will be replaced by the star Bernoulli
numbers $B_n^\star$'s.
\end{rmk}

\section{The case of general points with integral coordinates}\label{gen-point}
Before going to the general case, where $(a_1,\ldots,a_r)$
is any point in $\Z^r$, let us look at the special case of the 
Riemann zeta function. We can, for example, compute the regularised
value $\gamma_1^{(0)}$ of $\sum_{n\ge 1} \log n$, by using
the Stirling formula, as $N$ tends to $\infty$,
$$
\sum_{0 < n < N} \log n = \log (N-1)! = N\log N-N- \frac{1}{2} \log N
+ \frac{1}{2} \log 2\pi + o(1),
$$
so that $\gamma_1^{(0)}=\frac{1}{2} \log 2\pi$. In fact, we have
$\z'(0)=-\gamma_1^{(0)}$. More generally, as we shall see,
we have $D^k \z(0)= (-1)^k \gamma_k^{(0)}$ for $k \ge 1$.
However, this formula does not hold for $k=0$, since $\z(0)=-\frac{1}{2}$
and $\gamma_0^{(0)}=-1$. Similarly, at $-1$, we have
$D^k \z(-1)= (-1)^k \gamma_k^{(-1)}$ for $k \ge 2$,
but not for $k=0,1$. Our next theorem will allow us to understand these features,
even in the broader context of multiple zeta functions.

\begin{thm}\label{thm-gen-reg-exp}
Let $r \ge 0$ be an integer and $(a_1,\ldots,a_r) \in \Z^r$.
The power series \eqref{reg-multi-zeta}
$$
\sum_{k_1,\ldots,k_r \ge 0}
\frac{(-1)^{k_1+\cdots+k_r}}{k_1! \cdots k_r!}
\gamma_{k_1, \ldots, k_r}^{(a_1,\ldots,a_r)}
(s_1-a_1)^{k_1} \cdots (s_r-a_r)^{k_r},
$$
converges in a neighbourhood of $(a_1, \ldots, a_r)$
and extends to a meromorphic function in the whole of $\C^r$, denoted by
$\z^\reg_{(a_1, \ldots, a_r)} (s_1,\ldots,s_r)$.
Then we have the following equality
\begin{equation}\label{gen-reg-exp}
\begin{split}
\z^\reg_{(a_1, \ldots, a_r)} (s_1,\ldots,s_r) =
& \ \sum_{i=0}^r (-1)^i \z(s_{i+1},\ldots,s_r)
\sum_{k_1,\ldots,k_i \ge -1 \atop \sum\limits_{1 \le j \le i}(k_j +a_j)=0}
\frac{B_{k_1+1}^\star}{(k_1+1)!} \cdots \frac{B_{k_i+1}^\star}{(k_i+1)!} \\
& \times (s_i)_{k_i} (s_i+s_{i-1}+k_i)_{k_{i-1}} \cdots
(s_i+\cdots+s_1+k_i+\cdots+k_2)_{k_1},
\end{split}
\end{equation}
between meromorphic functions on $\C^r$.
\end{thm}

\begin{exm} \rm
Putting $r=a_1=1$, we first recover \eqref{zeta-exp}.
When $r=1$ and $a_1=-n$ for an integer $n \ge 0$, we further have
$$
\z^\reg_{(-n)} (s) = \z(s) - \frac{B_{n+1}^\star}{(n+1)!} (s)_n.
$$
Now for the Pochhammer symbol, we have for $n \ge 0$,
$$
(s)_n=\sum_{k=0}^n (-1)^{n-k} {\bf s}(n,k) s^k=\sum_{k=0}^n {\bf s}(n+1,k+1) (s+n)^k,
$$
where ${\bf s}(n,k)$'s are the Stirling numbers of the first kind,
defined by $(-1)^{n-k}\left[n \atop k\right]$, where
$\left[n \atop k\right]$ denotes the number of permutations
of $n$ elements with $k$ disjoint cycles. In particular, ${\bf s}(0,0) :=1$.
This means that the Taylor expansion of the Riemann zeta
function at $-n$ is of the form $\sum_{k \ge 0}a_k(s+n)^k$, where
$$
a_k=\begin{cases}
    \frac{(-1)^k \gamma_k^{(-n)}}{k!} + {\bf s}(n+1,k+1)\frac{B_{n+1}^\star}{(n+1)!}
    & \text{ when } k \le n,\\
    \frac{(-1)^k \gamma_k^{(-n)}}{k!} & \text{ when } k > n.
    \end{cases}
$$
\end{exm}

\begin{exm} \rm
We also exhibit another interesting example when $r=2$ and $a_1=a_2=0$.
In this case we have
\begin{equation*}
\begin{split}
\z^\reg_{(0,0)} (s_1,s_2)
&= \z(s_1,s_2) - B_1^\star \z(s_2)
+ \frac{B_0^\star B_2^\star}{2!} \frac{s_2}{s_1+s_2} + (B_1^\star)^2
+ \frac{B_0^\star B_2^\star}{2!} \frac{s_1+s_2-1}{s_2-1}\\
&= \z(s_1,s_2) - B_1^\star \z^\reg_{(0)}(s_2)
+ \frac{B_0^\star B_2^\star}{2!} \left( \frac{s_2}{s_1+s_2}
+ \frac{s_1+s_2-1}{s_2-1}\right).
\end{split}
\end{equation*}
Hence
$$
\z(s_1,s_2)=\z^\reg_{(0,0)} (s_1,s_2) + \frac{1}{2} \z^\reg_{(0)}(s_2)
- \frac{1}{12} \left(\frac{s_2}{s_1+s_2} + \frac{s_1+s_2-1}{s_2-1} \right).
$$
Since $\z^\reg_{(0,0)} (0,0)=\gamma_{0,0}^{(0,0)}=1$ and
$\z^\reg_{(0)}(0)=\gamma_{0}^{(0)}=-1$, we deduce from
the above formula that
$$
\lim_{s \to 0} \z(s,0)=\frac{5}{12} \ \text{ and } \ \lim_{s \to 0} \z(0,s)=\frac{1}{3}.
$$
The last formula can be found in \cite{AET}.
\end{exm}

\begin{proof}[Proof of \thmref{thm-gen-reg-exp}]

Recall \eqref{comb-form-1} that for an integer $N \ge 1$, we have the following equality between
meromorphic functions on $\C^r$ :
$$
\z(s_1,\ldots,s_r)_{< N}= \sum_{i=0}^r
(-1)^i \z^\star (s_i,\ldots,s_1)_{\ge N} \z(s_{i+1},\ldots,s_r),
$$
where
$\z(s_1,\ldots,s_r)_{<N} := \sum_{N > n_1 > \cdots >n_r>0} n_1^{-s_1} \cdots n_r^{-s_r}$
is a holomorphic function on $\C^r$ and for $(s_1,\ldots,s_r) \in U_r$,
$\z^\star(s_1,\ldots,s_r)_{\ge N} := \sum_{n_1 \ge \cdots \ge n_r \ge N}
n_1^{-s_1} \cdots n_r^{-s_r}$
which has a meromorphic extension to $\C^r$. We assume $N \ge 2$.

From \eqref{comb-form-1}, Remarks \ref{rmk-variable} and \ref{rmk-asymp-multi-zeta-star-tail},
we get that the sequence of germs at ${\bf a}=(a_1,\ldots,a_r)$ of holomorphic
functions $(\z(s_1,\ldots,s_r)_{< N})_{N \ge 2}$ has a complete
asymptotic expansion relative to $\cE$, in the sense of Section
\ref{asymp-mero}, and therefore also in the sense of Section \ref{asymp-holo},
by \rmkref{rmk-defn}. This shows that if we denote the associated formal complete
asymptotic expansion by $G=\sum_{l \in \N, m \in \Z} g_{(l,m)} L^l X^m$,
then by definition, $g_{(l,m)}$ is the germ at ${\bf a}$ of a holomorphic
function defined in a neighbourhood of ${\bf a}$.

This implies, in particular, that for any $k_1,\ldots,k_r \ge 0$, the sequence of
${\bf k}=(k_1,\ldots,k_r)$-th Taylor coefficients of the sequence of germs at ${\bf a}$
of holomorphic functions $(\z(s_1,\ldots,s_r)_{< N})_{N \ge 2}$, given by
$$
\left(\frac{(-1)^{k_1+\cdots+k_r}}{k_1! \cdots k_r!}
\sum_{N > n_1 > \cdots >n_r>0} \frac{\log^{k_1}n_1 \cdots \log^{k_r} n_r}
{n_1^{a_1} \cdots n_r^{a_r}} \right)_{N \ge 2},
$$
has an asymptotic expansion to arbitrary precision relative to $\cE$.
This result has already been proved independently in \thmref{thm-msc}.
The corresponding formal asymptotic expansion is therefore given by
$\sum_{l \in \N, m \in \Z} c_{\bf k}(g_{(l,m)}) L^l X^m$, where $c_{\bf k}(g_{(l,m)})$ is the
${\bf k}$-th Taylor coefficient of $g_{(l,m)}$ at ${\bf a}$.
Thus in particular, $c_{\bf k}(g_{(0,0)})=\frac{(-1)^{k_1+\cdots+k_r}}{k_1! \cdots k_r!}
\gamma_{k_1, \ldots, k_r}^{(a_1,\ldots,a_r)}$, and hence we get that the power series
\eqref{reg-multi-zeta} converges in a neighbourhood of ${\bf a}$ to a function
whose germ at ${\bf a}$ is $g_{(0,0)}$.

Using Remarks \ref{rmk-variable} and \ref{rmk-asymp-multi-zeta-star-tail}
in \eqref{comb-form-1}, we deduce that $G$ is the Laurent series
$$
\sum_{i=0}^r \sum_{k_1,\ldots,k_i \ge 0} \sum_{l \ge 0} h_{(i,l,k_1,\ldots,k_i)} L^l X^{i-a_1-\ldots-a_i-k_1-\ldots-k_i},
$$
where $h_{(i,l,k_1,\ldots,k_i)}$ is the germ at ${\bf a}$ of the function
\begin{align*}
(s_1,\ldots,s_r) \mapsto &
\frac{(-1)^{i+l}B^\star_{k_1} \cdots B^\star_{k_i}}{l! k_1!\cdots k_i!} \z(s_{i+1},\ldots,s_r)
(s_i)_{k_i-1} (s_i+s_{i-1}+k_i-1)_{k_{i-1}-1} \cdots \\
& \times (s_i+\cdots+s_1+k_i+\cdots+k_2-i+1)_{k_1-1} (s_1+\cdots+s_i-a_1-\cdots-a_i)^l.
\end{align*}
This gives that $g_{(0,0)}$ is nothing but the germ at ${\bf a}$ of the
meromorphic function
\begin{align*}
(s_1,\ldots,s_r) \mapsto & \ \sum_{i=0}^r (-1)^i \z(s_{i+1},\ldots,s_r)
\sum_{k_1,\ldots,k_i \ge -1 \atop \sum\limits_{1 \le j \le i}(k_j +a_j)=0}
\frac{B_{k_1+1}^\star}{(k_1+1)!} \cdots \frac{B_{k_i+1}^\star}{(k_i+1)!} \\
& \times (s_i)_{k_i} (s_i+s_{i-1}+k_i)_{k_{i-1}} \cdots
(s_i+\cdots+s_1+k_i+\cdots+k_2)_{k_1}.
\end{align*}
Therefore, the power series in \eqref{reg-multi-zeta} extends to a meromorphic
function in the whole of $\C^r$, satisfying \eqref{gen-reg-exp}. This completes the proof.
\end{proof}

\begin{rmk}\label{multi-zeta-star-expansion}
\rm We can analogously define {\it multiple star Stieltjes constants} 
$\gamma_{k_1, \ldots, k_r}^{\star(a_1,\ldots,a_r)}$ by considering
$\sum_{N \ge n_1 \ge \cdots \ge n_r \ge 1} \frac{\log^{k_1}n_1 \cdots \log^{k_r} n_r}
{n_1^{a_1} \cdots n_r^{a_r}}$, in place of $u_N$ in \thmref{thm-msc} and
Definition \ref{defn-msc}. Then we can consider the following formal power series
\begin{equation}\label{reg-multi-zeta-star}
\sum_{k_1,\ldots,k_r \ge 0}
\frac{(-1)^{k_1+\cdots+k_r}}{k_1! \cdots k_r!}
\gamma_{k_1, \ldots, k_r}^{\star(a_1,\ldots,a_r)}
(s_1-a_1)^{k_1} \cdots (s_r-a_r)^{k_r},
\end{equation}
which we denote by $\z^{\star\reg}_{(a_1, \ldots, a_r)} (s_1,\ldots,s_r)$.
Following the proof of \thmref{thm-gen-reg-exp} and using \eqref{comb-form-2}
in place of \eqref{comb-form-1}, we get that the power series \eqref{reg-multi-zeta-star}
converges in a neighbourhood of $(a_1, \ldots, a_r)$
and extends to a meromorphic function in the whole of $\C^r$ satisfying the equality
\begin{equation}\label{gen-reg-exp-star}
\begin{split}
\z^{\star\reg}_{(a_1, \ldots, a_r)} (s_1,\ldots,s_r) =
& \ \sum_{i=0}^r (-1)^i \z^\star(s_{i+1},\ldots,s_r)
\sum_{k_1,\ldots,k_i \ge -1 \atop \sum\limits_{1 \le j \le i}(k_j +a_j)=0}
\frac{B_{k_1+1}}{(k_1+1)!} \cdots \frac{B_{k_i+1}}{(k_i+1)!} \\
& \times (s_i)_{k_i} (s_i+s_{i-1}+k_i)_{k_{i-1}} \cdots
(s_i+\cdots+s_1+k_i+\cdots+k_2)_{k_1}
\end{split}
\end{equation}
of meromorphic functions on $\C^r$.

\end{rmk}

\bigskip

\noindent {\bf Acknowledgements  :}
I am highly thankful to Prof. Joseph Oestel\'e for numerous indispensable inputs
during this work and formation of this article. I am also thankful to him for
introducing me to the language of asymptotic expansion of sequences of
complex numbers and of germs of holomorphic and meromorphic functions,
relative to a comparison scale. This work was carried out in 2017-18, in
Institut de Math\'ematiques de Jussieu, with support from IRSES Moduli and LIA.

\end{document}